\documentclass{amsart}
\usepackage{dino}
\title{Scott Analysis below the Vaught Ordinal}
\author{David Gonzalez, Dino Rossegger, Dan Turetsky}

\address[Gonzalez]{University of Notre Dame\\
Department of Mathematics\\
Hurley Hall, 255 Hurley, Notre Dame, IN 46556\\
  USA}
\email{dgonza42@nd.edu}

\address[Rossegger]{Institute of Discrete Mathematics and Geometry, Technische Universit\"at Wien}
\email{dino.rossegger@tuwien.ac.at}

\address[Turetsky]{School of Mathematics and Stastics, Victoria University of Wellington}
\email{dan.turetsky@vuw.ac.nz}

\thanks{The authors would like to thank the Erwin Schrödinger International Institute for Mathematics and Physics (ESI) in Vienna for its hospitality and support during the thematic program 'Reverse Mathematics'. The first and second authors would like to thank the Hausdorff Research Institute for Mathematics (HIM) in Bonn for its hospitality during the Trimester Program 'Definability, Decidability, and Computability'. This work was supported by the Hausdorff Center for Mathematics, funded by the Deutsche Forschungsgemeinschaft (DFG, German Research Foundation) under Germany’s Excellence Strategy – 390685813. The work of the second author was supported by the Austrian Science Fund (FWF) 10.55776/PIN1878224 and 10.55776/PAT4699225. The work of the third author was supported by the Royal Society of New Zealand through the Marsden grant 23-VUW-118.}
\begin{document}
\begin{abstract}
We develop new tools for determining the existence of models of specific Scott ranks under countability conditions. Using these, we improve a result of Sacks by showing that any counterexample to Vaught’s conjecture must have at least two models of every parameterized Scott rank—a result that contrasts with the unparameterized case, where minimal counterexamples have only one model at many ranks. We further prove that theories with fewer than continuum many models have trivial Scott spectra and provide a general, systematic classification of low Scott rank models when only countably many $\Sigma_\alpha$-types are realized. Additionally, we classify the Scott complexity spectra for many Ehrenfeucht theories, and prove the $\omega$-Vaught’s conjecture in this setting, answering an infinitary strengthening of a question of Pillay and Tanović. We demonstrate that the Scott complexity of prime models for $\omega$-stable first-order theories is commensurate with the complexity of the theory itself. Along the way, we apply our methods to concrete theories like p-groups, trees, and Boolean algebras, answering questions of Harris--Montalbán and Alvir--Csima--MacLean regarding specific structures.
\end{abstract}
\maketitle
Much of contemporary mathematics comes down to investigating the properties of structures that satisfy some list of axioms, what logicians call a theory.
One of the most fundamental of these properties is about counting:
If you write down a theory, how many structures of a fixed size satisfy it?
While investigating theories with only finitely many models, Robert Vaught conjectured that the number of countable models of a theory is countable or continuum, but never in between \cite{Vau61}.
In other terms, unlike general sets, the conjecture states that the countable isomorphism types of a fixed theory satisfy the continuum hypothesis without further set theoretic assumptions.
This problem has been open for over sixty years and continues to inspire work in various parts of mathematical logic, including model theory, descriptive set theory, and computability theory.
It is unclear what part of mathematical logic will resolve Vaught's conjecture, or what further interactions its continued investigation will promote.

Important early progress on Vaught's conjecture was achieved by Morley, who demonstrated that the number of countable models of a theory is countable, continuum, or exactly $\aleph_1$, the second smallest cardinal, and never anything else \cite{Mor70}.
Morley's approach centrally informs many contemporary approaches to Vaught's conjecture and similar model counting problems, including the work in this article.
Morley considered the \textbf{infinitary Vaught's conjecture}, a stronger version that lets theories be axiomatized by a formula in infinitary logic, $L_{\omega_1\omega}$, instead of just a first-order theory.
This is the setting with which we presently concern ourselves; henceforth ``Vaught's conjecture'' refers to the infinitary version.
This setting has the advantage that one can exploit properties of infinitary sentences not satisfied by finitary theories, for example, Scott's theorem that every countable structure is characterized up to isomorphism by an $L_{\omega_1\omega}$ sentence~\cite{Sco65}.
Scott's work allows for several invariants to be associated with a countable structure, all of which roughly measure the difficulty of describing the isomorphism type of the structure.
These invariants are often called the \emph{Scott rank} or \emph{Scott complexity} of the structure.
The field of Scott analysis aims to understand how these invariants relate to other important notions throughout logic, how to effectively calculate Scott invariants, and which sorts of Scott invariants emerge from models of a particular theory as a means of comparing them.
Scott analysis has emerged as an independently interesting field touching various parts of logic in addition to its usefulness for problems like Vaught's conjecture (see the survey~\cite{HT22} for many examples and more background on Scott analysis).

Sacks considered Morley's analysis closely and noted that every theory satisfying Vaught's conjecture has an associated countable ordinal that witnesses this satisfaction~\cite{Sac07}.
The \emph{Vaught ordinal} of a theory, $vo(T)$, is the least ordinal that witnesses one of two behaviors.
\begin{enumerate}
 \tightlist
 \item The first possibility, corresponding to the continuum outcome, is that there are continuum many different formulas true of the various models of $T$ at level $vo(T)$.
 \item The second possibility, corresponding to the countable outcome, is that there are no models of Scott rank greater than $vo(T)$.
\end{enumerate}
The existence of a countable Vaught ordinal for $T$ is equivalent to $T$'s satisfaction of Vaught's conjecture, whereas a counterexample to Vaught's conjecture has $vo(T)=\omega_1$.

The present article concerns Scott analysis below the Vaught ordinal of a theory.
In other words, we consider theories $T$ and ordinals $\alpha$ such that $T$ has models of Scott rank at least $\alpha$, but there are only countably many different $\Pinf{\alpha}$ types realized by the models of $T$.
There are two ways to understand this effort.
The first is that we are attempting to understand Vaught's conjecture.
A counterexample to Vaught's conjecture, should it exist, would have $vo(T)=\omega_1$, so our analysis would apply to all levels of the quantifier-complexity hierarchy.
Through this effort, we obtain a better understanding of counterexamples to Vaught's conjecture, whether in an attempt to construct one or an attempt to disprove their existence.
The second is to develop general Scott analytic tools for concrete theories of interest below their Vaught ordinal.
For example, the work of Harris and Montalbán \cite{HM} calculates that the Vaught ordinal of Boolean algebras is $\omega$ and describes the structure of Boolean algebras up to finitely many alternations of quantifiers.
Similarly, work of the first two authors \cite{GR24} explained the behavior of linear orderings below their Vaught ordinal of 3.
The present work provides tools for similar future efforts within a more abstract, widely applicable framework.
We understand our work in both contexts and emphasize applications in each framework.

We use notions of Scott rank introduced by Montalbán in \cite{MonSR} that we more precisely define in the preliminaries section.
More specifically, we refer to the \emph{(unparamaterized) Scott rank} of a structure, $SR(\M)$, and the \emph{parametrized Scott rank} of a structure, $SR_p(\M)$.
The difference between these notions is that the parameterized Scott rank allows for finitely many parameters to be named to aid in the description of the isomorphism type of the structure (e.g., a basis of a finite-dimensional vector space).
These notions relate closely to \emph{Scott complexity}, or the syntactic complexity of the simplest Scott sentence of a structure.

We take care to note the variety of Scott invariants used because, unlike most previous work, we observe qualitative differences between the behavior of theories based on these details.
This is most apparent when it comes to counting models of a particular Scott rank in a hypothetical counterexample to Vaught's conjecture.
This line of inquiry is due to Sacks \cite{Sac83}, who identifies two reasons to understand Vaught's conjecture from this perspective.
The first is that there are counterexamples to strong versions of Vaught's conjecture that allow more expressive logics in describing the theory and have very particular shapes from the perspective of Scott rank.
Sacks, giving credit to H. Friedman, describes the set of $1$-transitive linear orderings (those for which any pair of singletons is in the same automorphism orbit) as a $PC_{\omega_1,\omega}$ class with exactly $\aleph_1$ models, yet there is only (exactly) one such model of each Scott rank.
Showing that this sort of Scott analytic behavior is impossible in $L_{\omega_1,\omega}$ demonstrates the futility of adapting such examples to a true $L_{\omega_1,\omega}$ counterexample.
The second reason to study this is that a possible route to proving Vaught's conjecture is to consider a hypothetical counterexample and show that it must have far too many models at a given Scott rank.
Any progress towards ``thickening'' the profile of a theory that is a counterexample to Vaught's conjecture by demonstrating the necessity of many models works towards this goal.
Sacks himself in \cite{Sac83} demonstrates that there must be multiple models of each $\Sigma_2^1$-admissible Scott rank, and this has not been previously improved since his original paper.
We offer the following improvement.
\begin{restatable*}{corollary}{twomodelsvc}\label{cor:2ModelsVC}
If $T$ is a $\Pinf{2}$ counterexample to Vaught's conjecture, it has at least two models of every parameterized Scott rank.
\end{restatable*}
We also demonstrate limits to this methodology by showing that if Vaught's conjecture fails, then there must be a counterexample with exactly one model of unparameterized Scott rank $\alpha$ for $\aleph_1$-many $\alpha$ (\cref{thm:minimalcealeph1}).
That said, this contrast offers a different approach to proving Vaught's conjecture.
Although distinct, the parameterized and unparameterized Scott ranks are very closely related.
If it were possible to strengthen one of our proofs to show that either a) counterexamples to Vaught's conjecture have at least two models of every unparameterized Scott rank or b) there must be a counterexample with exactly one model of each parameterized Scott rank, then Vaught's conjecture would be proven by a direct contradiction.
Perhaps it is even possible to find a notion of rank that sits between parameterized and unparameterized Scott rank that would allow such a conclusion.
It is tantalizing to consider; the only thing stopping this proof idea from going through currently is the influence of a few parameters.
That said, it is not clear how to remove those parameters for now.

In order to prove these results, we primarily use three technical tools.
The first two use the assumption of $\Sinf{\alpha}$-smallness, or that there are only countably many $\Sinf{\beta}$ types realized among models of $T$.
\begin{restatable*}{lemma}{havingPimodels}\label{lem:havingPimodels}
  If $T$ is a consistent theory that is $\Sinf{\beta}$-small and $T\in \Pinf{\beta+1}$, then there is a model $\A\models T$ so that $SR(\A)\leq \beta$. 
\end{restatable*}
\begin{restatable*}{lemma}{havingSigmamodels}\label{lem:havingSigmamodels}
  If $T$ is a consistent theory that is $\Sinf{\beta}$-small and $T\in \Sinf{\beta+2}$, then there is a model $\A\models T$ so that $SR_p(\A)\leq \beta$. 
\end{restatable*}
The third tool does not depend on counting assumptions.
It states, using a notion fully explained in the preliminaries section, that if a Scott rank $\alpha$ model $\A$ shares all $\Sinf{\alpha}$ and $\Pinf{\alpha}$ facts with another model $\B$, it also shares all  $\Pinf{\alpha+1}$ facts with $\B$.
\begin{restatable*}{proposition}{uniqueminSR}\label{prop:uniqueminSR}
    Let $SR(\mathcal{A})\leq\alpha$ and assume that $\mathcal{A}\equiv_{\alpha}\mathcal{B}$. Then, $\mathcal{A}\geq_{\alpha+1}\mathcal{B}$.
\end{restatable*}

These results find many other applications in the broader Scott analysis literature.
For example, we prove that, below the Vaught ordinal, there is no ``Scott skipping'', or omitted Scott ranks, in stark contrast to the general case where Harrison-Trainor showed that arbitrarily long intervals of ordinals can be skipped~\cite{HT18}.

\begin{restatable*}{theorem}{ssptrivial}\label{thm:ssptrivial}
	Suppose $T$ is a consistent $\Pinf{\alpha}$ sentence with less than $2^{\aleph_0}$ countable models. Then $T$ has a model of Scott rank $\beta$ for all $\beta\in [\alpha,vo(T))$. Thus, the Scott spectrum of $T$ is trivial.
\end{restatable*}
We demonstrate that, below the Vaught ordinal, there are always models with different parameterized and unparamaterized Scott ranks.
This allows us to give an entirely non-constructive affirmative answer to a recent question of Alvir, Csima, and MacLean about p-groups \cite{ACM24}.
\begin{restatable*}{corollary}{abelianpsigma}\label{cor:abelianpsigma}
There is an Abelian p-group with Scott complexity $\Sinf{3}$, or there are infinitely many with Scott complexity $\dSinf{2}$.
\end{restatable*}
We also get new insights into the theory of order-theoretic trees, a theory first understood in the context of Scott analysis by Steel \cite{Ste78}, by describing a general method of corresponding a $\Pinf{\alpha}$ type with models of Scott rank $\alpha$ below the Vaught ordinal.
In particular, we show that the simplest order-theoretic trees correspond to the finitely many finite trees that they (minimally) do not take embeddings from.
These tools also aid in solving a conjecture left open by Harris and Montalbán regarding Boolean algebras \cite{HM}.
We demonstrate that any $\Sinf{n}$-type of Boolean algebras that only extends to a single $\Sinf{n+1}$-type describes the isomorphism type of a Boolean algebra of finite Scott rank.

We pay special attention to \define{Ehrenfeucht theories}; theories with only finitely many models.
As mentioned, these theories have a special historical tie with Vaught's conjecture, as Vaught's original question was posed in a paper considering these very theories~\cite{Vau61}.
We consider infinitary Ehrenfeucht theories, unlike Vaught.
In their write-up of Vaught's conjecture and related open problems~\cite{PT25}, Pillay and Tanović emphasize the importance of Ehrenfeucht theories.
Among the questions they ask about Ehrenfeucht theories is whether they satisfy Martin's conjecture~\cite[Question 7]{PT25}.
Martin's conjecture is a model-theoretic strengthening of Vaught's conjecture, which, in the infinitary setting, translates to what Gonzalez and Montalbán called the $\omega$\textbf{-Vaught conjecture} ($\omega$-VC)~\cite{GM23}.
It states that, not only does every theory have a countable Vaught ordinal, but every $\Pinf{\alpha}$ theory has a Vaught ordinal at most $\alpha+\omega$.
In the setting of infinitary logic, we answer Pillay and Tanović's question in the affirmative.
\begin{restatable*}{corollary}{ehrenfeuchtomegavaught}\label{cor:ehrenfeuchtomegavaught}
For any $\Pinf{\alpha}$ Ehrenfeucht theory $T$, $vo(T)\leq\alpha+\omega$. In the language of \cite{GM23}, Ehrenfeucht theories satisfy $\omega$-Vaught's conjecture.
\end{restatable*}

Our analysis of Ehrenfeucht theories is more precisely calibrated than what is represented above.
In fact, we show that, for every $n$, among all theories with exactly $n$ models, there are only finitely many possible arrangements of the Scott complexities realized by those models. 
(We call these arrangements \emph{Scott complexity spectra} as introduced in \cite{GLRS}.)
Letting $F(n)$ be the number of these spectra on $n$ models, we calculate that $F(2)=2$ and $F(3)=4$ by theoretically cutting out most possible spectra and explicitly constructing theories that exhibit the remaining behavior.
The value of $F(n)$ for higher $n$ remains an open area for future research.

The last application of our tools is to analyze prime models of first-order theories in the infinitary setting of Scott analysis.
The main result in this area is as follows.
\begin{restatable*}{theorem}{pinprime}\label{thm:pinprime}
Let $T$ be a complete $\Pi_{n+2}$ axiomatizable first-order theory that is $\omega$-stable or has fewer than continuum many countable models.
The prime model $P$ must have Scott complexity at most $\Pinf{n+2}$.
\end{restatable*}

Under suitable countability assumptions, we obtain that the prime model is only as complicated as the theory it corresponds to.

The paper is organized into five further sections.
We begin with preliminaries that establish notation, precise definitions for key notions, and critical results used consistently throughout the paper.
In Section 2, we prove our technical tools, demonstrating the existence of low Scott rank models given the correct countability assumptions.
We apply these tools to prove that any counterexample to Vaught's conjecture has two models of each parameterized Scott rank and that theories have models of every rank up to their Vaught ordinal.
Section 3 broadens the technical toolkit by first proving that being $\alpha$-equivalent to a Scott rank $\alpha$ model implies being $\alpha+1$-below it.
Using this tool, we show that, given the negation of Vaught's conjecture, there are counterexamples to Vaught's conjecture with precisely one model of some unparamaterized Scott rank in $(\alpha,\omega_1)$ for each $\alpha$.
We then investigate models with different parameterized and unparameterized Scott ranks in the presence of countability and answer a p-group question of Alvir, Csima, and MacLean~\cite{ACM24}.
Finally, correspondence between $\Pinf{\alpha}$-types and Scott rank $\alpha$ models is established and used to investigate order-theoretic trees and their properties.
Section 4 is focused entirely on Ehrenfeucht theories.
We prove $\omega-$VC for Ehrenfeucht theories and perform a finer analysis of the possible Scott complexity spectra in this section.
The final section, Section 5, focuses on prime models of first-order theories and establishes our result on the complexity of the model versus the complexity of the theory.

\section{Preliminaries}
When speaking about $L_{\omega_1\omega}$ formulas, a critical notion is that they have a normal form where all of the quantifiers and infinitary connectives are placed at the front of the formula.
Using this normal form, we can describe the complexity of a formula as follows.

\begin{definition}
	We let $\Pinf{0}=\Sinf{0}$ be the class of finitary quantifier-free formulas. For countable ordinals $\alpha>0$ and $\phi\in L_{\omega_1\omega}$ we let 
	\begin{itemize}
		\item $\phi\in \Sinf{\alpha}$ if it is logically equivalent to a formula of the form $\bigvvee_i \exists \bar x\psi_i(\bar x)$ where $\psi_i\in \Pinf{\gamma_i}$ for some $\gamma_i<\alpha$;
		\item $\phi\in \Pinf{\alpha}$ if it is logically equivalent to a formula of the form $\bigwwedge_i \forall\bar x\psi_i(\bar x)$ where $\psi_i\in \Sinf{\gamma_i}$ for some $\gamma_i<\alpha$.
	\end{itemize}
	Furthermore, $\phi\in \dSinf{\alpha}$ if there is $\psi\in \Sinf{\alpha}$ and $\theta\in \Pinf{\alpha}$ so that $\phi$ is logically equivalent to $\psi\land \theta$.
\end{definition}

Note that the negation of a $\Sinf{\alpha}$ formula is a $ \Pinf{\alpha}$ formula and vice-versa.
Also, $\Sinf{\alpha}, \Pinf{\alpha}\subseteq \dSinf{\alpha}\subseteq \Sinf{\alpha+1}, \Pinf{\alpha+1}$.
A fundamental theorem of Scott \cite{Sco65} is that every countable structure has an $L_{\omega_1\omega}$ formula that describes it up to isomorphism among countable structures.
Such a formula for a structure $\M$ is called a \textbf{Scott sentence} for $\M$.
This gives us a way to associate formulas with structures.
The ranking on formulas, therefore, gives us a ranking on structures by association.
This ranking generally measures the difficulty of describing the structure up to isomorphism.
To be more precise, we will work with three variations of this rank.
Montalban defined the first two in \cite{MonSR}.
The \textbf{(unparamaterized) Scott rank} of a structure, $SR(\M)$, is given by the least $\alpha$ such that $\M$ has a $\Pinf{\alpha+1}$ Scott sentence.
Meanwhile, the \textbf{paramaterized Scott rank} of a structure, $SR_p(\M)$, is given by the least $\alpha$ such that $\M$ has a $\Sinf{\alpha+2}$ Scott sentence.
The parameterized Scott rank can also be understood as the minimal unparamaterized Scott rank of $\M$ after taking some fixed finite set of parameters.
Both of these measures are robust in the sense that there are many other equivalent formulations of these notions.
Notable for our purposes is that $SR(\M)$ is also the least $\alpha$ such that every automorphism orbit in $\M$ is $\Sinf{\alpha}$ definable.
This is further equivalent to saying that the $\Pinf{\alpha}$ type of each tuple is supported by a $\Sinf{\alpha}$ formula. 
Similarly, $SR_p(\M)$ is the least $\alpha$ such that every automorphism orbit in $\M$ is $\Sinf{\alpha}$ definable over some fixed finite set of parameters.
We lastly refer to \textbf{Scott complexity}, $SC(\M)$, a finer invariant introduced in \cite{AGNHTT}.
There it is shown that every structure $\M$ has a simplest Scott sentence of the form $\Sinf{\alpha}, \Pinf{\alpha}$ or $\dSinf{\alpha}$ for some ordinal $\alpha$.
The Scott complexity of $\M$ is the complexity among $\Sinf{\alpha}, \Pinf{\alpha}$ and $\dSinf{\alpha}$ of the simplest Scott sentence for $\M$.
We note that infinite structures have Scott complexity $\Pinf{2}$ at the minimum \cite[Theorem 5.1]{AGNHTT,miller1983}.

We also frequently use the notion of the \textbf{Vaught ordinal} of a theory, in the vein of \cite{GM23}.
Given a theory $T\in  L_{\omega_1\omega}$, we let $vo(T)$ be the least ordinal such that either
\begin{itemize}
	\item there are continuum many $\Sinf{\gamma}$ types realized among the models of $T$
	\item there are countably many models of $T$ and they all have Scott rank less than $\beta$.
\end{itemize}
The Vaught ordinal is so named because of its relationships to Vaught's conjecture.
In particular, it is not hard to show (see \cref{prop:smallvaughts}) that $T$ satisfies Vaught's conjecture if and only if $vo(T)<\omega_1$.

A useful tool to understand lower bounds for Scott rank or Scott complexity is the back-and-forth relations.
\begin{definition}
The \define{back-and-forth relations} $\leq_\alpha$, for a countable ordinal $\alpha < \omega_1$, are defined by:
    \begin{itemize}
        \item $(\mc{M},\bar{a}) \leq_0 (\mc{N},\bar{b})$ if $\bar{a}$ and $\bar{b}$ satisfy the same quantifier-free formulas from among the first $|\bar{a}|$-many formulas.
        \item For $\alpha > 0$, $(\mc{M},\bar{a}) \leq_\alpha (\mc{N},\bar{b})$ if for each $\beta < \alpha$ and $\bar{d} \in \mc{N}$ there is $\bar{c} \in \mc{M}$ such that $(\mc{N},\bar{b} \bar{d}) \leq_\beta (\mc{M},\bar{a} \bar{c})$.
    \end{itemize}
We define $\bar{a} \equiv_\alpha \bar{b}$ if $\bar{a} \leq_\alpha \bar{b}$ and $\bar{b} \leq_\alpha \bar{a}$.
\end{definition}

The relation $(\mc M,\bar a)\leq_\alpha (\mc N,\bar b)$ can be thought of as a game between the $\forall$-player who moves first and picks an ordinal  $\beta<\alpha$ and a tuple of elements $\bar d\in\mc{N}$ and the $\exists$-player who moves in response to this choice of tuple by picking a tuple $\bar c\in \mc{M}$ so that the quantifier-free formulas holding of $\bar b\bar d$ in $\mc N$ and $\bar a \bar c$ in $\mc M$ coincide. The game then continues with $(\mc N,\bar b\bar d)\leq_\beta (\mc M,\bar a\bar c)$ as the starting condition, i.e., the $\forall$-player plays in $\mc M$ and $\gamma<\beta$ and the $\exists$-player responds in $\mc N$. The play continues on like this until the $\forall$-player either runs out of ordinals to play, in which case the $\exists$-player wins, or the $\exists$-player responds with a tuple so that the quantifier-free formulas do not match, in which case the $\forall$-player wins.
The relation $(\mc M,\bar a)\leq_\alpha(\mc N,\bar b)$ holds if the $\exists$-player has a winning strategy in this game.
The following theorem connects these relations with infinitary logic.

\begin{theorem}[\cite{Karp}]
For any non-zero ordinal $\alpha$, structures $\mc{M}$ and $\mc{N}$ and tuples $\bar{a}\in\mc{M}$ and $\bar{b}\in\mc{N}$, the following are equivalent:
    \begin{enumerate}
	\item $(\mc{M},\bar{a})\leq_\alpha (\mc{N},\bar{b})$.
	\item Every $\Pinf{\alpha}$ formula true about $\bar{a}$ in $\mc{M}$ is true about $\bar{b}$ in $\mc{N}$.
	\item Every $\Sinf{\alpha}$ formula true about $\bar{b}$ in $\mc{N}$ is true about $\bar{a}$ in $\mc{M}$.
    \end{enumerate} 
\end{theorem}

There is a tight connection between winning strategies in the back-and-forth game and Scott rank through the notion of $\alpha$-freeness.

\begin{definition}[\cite{Part2} Chapter II.9]
A tuple $\ba$ is \define{$\alpha$-free} in the structure $\A$ if for every tuple $\bb\in A$ and every $\beta<\alpha$, there are tuples $\ba',\bb'$ such that 
\[\ba\bb\leq_\beta\ba'\bb' \text{ and } \ba\not\leq_\alpha\ba'.\]
\end{definition}

We think of the notion of $\alpha$-freeness as an adaptation of the back-and-forth game, where the $\forall$-player picks in round $0$ an additional parameter $\bar b$ and ordinal $\beta<\alpha$, and the $\exists$-player responds with elements $\bar a'$ and $\bar b'$. Then they play the game $\bar a\leq_\alpha \bar a'$ and $\bar a \bar b\leq_\beta \bar a'\bar b'$. The $\exists$-player wins if they lose the $\leq_\alpha$ game and win the $\leq_\beta$-game. Otherwise, the $\forall$-player wins. A tuple is now $\alpha$-free if the $\exists$-player has a winning strategy in this game.
\begin{theorem}[\cite{MonSR}]
The unparameterized Scott rank of a structure is the least $\alpha$ such that no tuple is $\alpha$-free.
The parameterized Scott rank of a structure is the least $\alpha$ such that no tuple is $\alpha$-free over some set of fixed, finite parameters.
\end{theorem}

\section{Consequences of $\Sinf{\alpha}$-smallness}

\begin{definition}\label{def:scottspectrum}
	Let $T$ be an $L_{\omega_1\omega}$ sentence. Then the \define{Scott spectrum} of $T$ is the set
	\[ SSp(T)=\{ SR(\A): \A\models T\}.\]
	For $T$ a $\Pinf{\alpha}$ sentence, we say that $SSp(T)$ is \define{trivial} if $T$ intersects $[\alpha,\omega_1)$ as an initial segment.
\end{definition}
One of the main results of this section is the following, which we will prove in \cref{sec:exactSC}.
\ssptrivial
Harrison-Trainor~\cite{HT18} studied possible Scott spectra of $L_{\omega_1\omega}$ sentences. The examples he produced all have continuum many models. \cref{thm:ssptrivial} shows that for all but the basic examples, this is necessary.

\subsection{Type omitting style results}
In this subsection, we prove two of the major technical tools used in this paper.
In particular, we show how to construct a model of bounded Scott rank given sufficient countability conditions. 
These insights are later used to produce concrete Scott analytic results, given assumptions that are based on counting.
We will use the following convention, introduced in \cite{MonICM}, to parsimoniously refer to theories without too many types realized among their models.

\begin{definition}\label{def:small}
  A theory $T$ is $\Sinf{\alpha}$-small if there are only countably many $\Sinf{\alpha}$ types realized among its models. 
\end{definition}

\begin{proposition}\label{prop:smallvaughts}
  Suppose $T$ is a counterexample to Vaught's conjecture, then $T$ is $\Sinf{\alpha}$ small for all $\alpha<\omega_1$.
\end{proposition}

\begin{proof}
Say that $T$ is not $\Sinf{\alpha}$ small for some $\alpha<\omega_1$.
Consider the equivalence relation on tuples in models of $T$ given by $\equiv_\alpha$.
This is a Borel equivalence relation, so by Silver's theorem \cite{Sil80}, it either has countably many or continuum many equivalence classes.
By assumption, it does not have countably many equivalence classes, so it has continuum many.
As every model of $T$ only has countably many elements, it can only realize countably many $\equiv_\alpha$ classes.
Therefore, there must be continuum many models of $T$, and so $T$ satisfies Vaught's conjecture.
\end{proof}

The following is a similar theorem to the main type omitting theorem of \cite{MonSR}.
That said, the construction has a different flavor, analogous to the existence of prime models in model theory rather than a first-order type omitting theorem.

\havingPimodels

\begin{proof}
Let $\B\models T$.
Enumerate the $\Pinf{\beta}$ types realized among the models of $T$ as $\{\Gamma_i(\bar{z}_i)\}_{i\in\omega}$.
We use a Henkin construction to create the desired $\A$.
Write $T$ as $\bigwwedge_{i\in\omega} \forall \bar{x}_i \phi_i(\bar{x}_i)$ with $\phi_i\in \Sinf{\beta}$.
We build a set of $\Sinf{\beta}$ formulas $S$ over the vocabulary $\tau$ of $\B$ enriched with countably many constants $C$.
We ensure that $S$ meets the following requirements, most of which are standard:
		\begin{enumerate}
			\item If $\bigvvee \psi_i \in S$, then for some $i$, $\psi_i \in S$.
			\item If $\exists \bar y \psi(\bar y) \in S$, then $\psi(\bar c) \in S$ for some constants $\bar c \in C$.
			\item If $\bigwwedge \psi_i \in S$, then for all $i$, $\psi_i \in S$.
			\item If $\forall \bar y \psi(\bar y) \in S$, then $\psi(\bar c) \in S$ for all $\bar c \in C$.
			\item For every atomic sentence $\psi$ over $\tau \cup C$, either $\psi \in S$ or $\neg \psi \in S$.
			\item For every $i$ and tuple $\bar{c}$ of length $\vert \bar{x}_i \vert$, $\varphi_i(\bar{c})\in S$.
			\item For every tuple $\bar{c}$ and $i\in\omega$ where $|\bar{c}|=|\bar{z}_i|$, there is either a $\theta\in \Gamma_i(\bar{z}_i)$ such that $\lnot\theta(\bar{c})\in S$ or there is some $\chi\in\Gamma_i(\bar{z}_i)$ with $\chi(\bar{c})\in S$ and $\chi(\bar{z}_i)$ semantically entails $\Gamma_i(\bar{z}_i)$ among models of $T$.
		\end{enumerate}
    We will build $S$ in stages, at each stage ensuring that one of the properties is satisfied (i.e., at stage $s$ we ensure that Property $(s\mod 7)+1$ is satisfied). Simultaniously we will make sure at each stage $s$ there is some interpretation $\nu_s:C\to \mc M_s \in Mod(T)$ that assigns the constants we have used so far to a set of elements in some model $\mc M_s$ of $T$ so that $\mc M_s\models S[\frac{C}{\nu_s(C)}]$.
		If we achieve this, then in the limit, the Henkin model $\A$ will satisfy all sentences in $S$ and thus also $T$ by item (6).
		Furthermore, we claim that every $\Pinf{\beta}$ type will be isolated by a $\Sinf{\beta}$ formula, so $SR(\A)\leq \beta$:
		In $\A$, because $\A\models T$, every tuple $\bar c\in \A$ will satisfy one of the $\Gamma_i$.
		Applying item (7) to the pair $\bar c,i$ yields two possibilities.
		The first is that for some $\theta\in\Gamma_i,$ $\A\models\lnot \theta(\bar{c})$; this is impossible as $\Gamma_i$ is assumed to the the $\Pinf{\alpha}$-type of $\bar{c}$.
		This leaves the second possibility, that there is some $\chi\in\Sinf{\beta}$ with $\A\models\chi(\bar{c})$ and for each $\mu\in\Gamma_i(\bar{z}_i)$, $\A\models\forall \bar x\chi(\bar x)\to \mu(\bar x).$
		This exactly states that the $\Pinf{\beta}$ type of $\bar{c}$ is $\Sinf{\beta}$ isolated, as desired.
		
		All that remains is to actually build the promised set of formulas.
		We begin with empty sets $U_0$, $C_0$, and $\nu_0$.
		At each stage we are given $C_s$ a finite subset of $C$, $S_s$, a finite set of  $\Sinf{\alpha}$ formulas only mentioning constants $C_s$, $\M_s$, a model of $T$, and $\nu_s:C_s\to \M_s$ with the property that $\M_s\models S_s(\nu_s(C_s))$.
		At each stage, we address one of the properties $(1)-(7)$, one instance at a time.
		We describe below how $S_s$, $C_s$, $\M_s$ and $\nu_s$ are modified to achieve this.
		For the most part (indeed, for properties $(1)-(6)$), the approach is common to the infinitary type omitting scheme.
		That said, we still include it for completeness.
		
    \begin{enumerate}[align=left,leftmargin=\parindent]
      \item[Property $(1)$]
			Given $\bigvvee \psi_i \in S_s$ we know that $\M_s\models \bigvvee \psi_i (\nu_s(C_s)))$. 
			Pick a $j$ with the property that $\M_s\models \psi_j (\nu_s(C_s)))$ and add $\psi_j$ to $S_s$ to make $S_{s+1}$.
			Let $\M_{s+1}=\M_s$, $\nu_s=\nu_{s+1}$ and $C_{s+1}=C_s$.
			\item[Property $(2)$]
			Given $\exists \bar y \psi(\bar y) \in S_s$ we know that  $\M_s\models  \exists \bar y\psi (\bar{y}, \nu_s(C_s)))$.
			Let $\bar{a}\in\M_s$ have the property that $\M_s\models \psi(\bar{a},\nu_s(C_s))$.
			Take new constants $\bar{d}\in C$, let $C_{s+1}=C_s\cup\{\bar{d}\}$, and $S_{s+1}= S_s\cup\{\psi(\bar{d})\}$.
			Let $\M_{s+1}=\M_s$, and extend $\nu_s$ by defining $\nu_{s+1}(\bar{d})=\bar{a}$.
			\item[Property $(3)$]
			Given $\bigwwedge \psi_i \in S_s$ and $i\in\omega$ let $S_{s+1}=S_s\cup\{\psi_i:i<s\}$, $\M_{s+1}=\M_s$, $C_{s+1}=C_s$ and $\nu_{s+1}=\nu_s$.
			\item[Property $(4)$]
        Given $\forall \bar y \psi(\bar y) \in S_s$, let $S_{s+1}=S_s\cup\{\psi(\bar c): \bar c\in C_s^{<|y|}\}$, $\M_{s+1}=\M_s$, $C_{s+1}=C_s$ and $\nu_{s+1}=\nu_s$.
			\item[Property $(5)$]
			For any atomic sentence $\psi$ and $\bar{c}\in C_s$ of the appropriate length check if $\M_s\models \psi(\nu_s(\bar{c}))$ or if $\M_s\models \lnot\psi(\nu_s(\bar{c}))$.
			In the former case, add $\psi(\bar{c})$ to $S_s$ to obtain $S_{s+1}$ and in the later case  add $\lnot\psi(\bar{c})$ to $S_s$ to obtain $S_{s+1}$.
			Either way, let $\M_{s+1}=\M_s$, $C_{s+1}=C_s$ and $\nu_{s+1}=\nu_s$.
			\item[Property $(6)$]
			Given $\varphi_i$ and $\bar{c}\in C_s$ of the appropriate length $S_{s+1}=S_s\cup\{\varphi_i(\bar{c})\}$, $\M_{s+1}=\M_s$, $C_{s+1}=C_s$ and $\nu_{s+1}=\nu_s$.
			\item[Property $(7)$]
			This is the key property argument and the only property that necessitates shifting $\nu_s$ and $\M_s$.
			Fix a tuple  $\bar{c}\in C_s$ of length $\vert \bar{z} \vert$ and let $\bar{d}$ be the elements of $C_s$ that are not among $\bar{c}$.
			We can write $\bigwedge S_s=\rho(\bar{c},\bar{d})$ where $\rho$ is a $\Sinf{\alpha}$ formula.
			Note that $\exists\bar{y}\rho(\bar{c},\bar{y})$ is also a $\Sinf{\alpha}$ formula that is satisfied in $\M_s$ by $\nu_s(\bar{c})$.
			Consider all possible $\bar{x}$ satisfying $\exists\bar{y}\rho(\bar{x},\bar{y})$ in any model of $T$.
			There are two possibilities.
			The first is that there is some $\theta\in\Gamma_i$ and some $\bar{d}\in \mathcal{N}$ and $\mathcal{N}\models T \land \exists\bar{y}\rho(\bar{d},\bar{y})\land \lnot \theta(\bar{d})$.
			Let $\bar e\in \mathcal{N}$ witness the existential quantifier in $\exists\bar{y}\rho(\bar{d},\bar{y})$.
			In this case, we let $C_{s+1}=C_s$, $S_{s+1}=S_s\cup\{\lnot\theta(\bar{c})\}$, $\M_{s+1}=\mathcal{N}$.
			Note that $\lnot\theta\in \Sinf{\beta}$ as $\theta\in \Pinf{\beta}$, so this is an allowable extension of $T$.
			Finally, let $\nu_{s+1}(\bar{c})=\bar{d}$ and let $\nu_{s+1}(\bar{d})=\bar{e}$ where $\mathcal{N}\models \rho(\bar{d},\bar{e})$.
			The other possibility is that for all $\theta\in\Gamma_i$ and $\mathcal{N}\models T$, $\mathcal{N}\models \forall \bar x\exists\bar{y}\rho(\bar{x},\bar{y})\to\theta(\bar{x})$.
			In this case, we may take $C_{s+1}=C_s$, $S_{s+1}=S_s\cup\{\exists\bar{y}\rho(\bar{c},\bar{y})\}$, $\M_{s+1}=\M_s$ and $\nu_{s+1}=\nu_s$.
			In terms of the outlined properties, we now have $\chi=\exists\bar{y}\rho(\bar{c},\bar{y})$ that holds of $\bar{c}$ that semantically entails every formula in $\Gamma_i$ among models of $T$.
		\end{enumerate}
		
		An easy induction shows that the Henkin model $\A$ obtained from the constants $C=\lim C_s$ will satisfy $S=\lim S_s$ and thus also properties (1)-(7), completing the proof.

\end{proof}

We now prove an analog of the above lemma for parameterized Scott rank.
Mostly, we can rely on the tools already established in the unparameterized setting.

\havingSigmamodels

\begin{proof}
Write $T=\bigvvee_{i}\exists \bar x_i\psi_i(\bar x_i)$ where each $\psi_i\in\Pinf{\beta+1}$.
Say that $\B$ witnesses the consistency of $T$, or, more specifically, $\B\models \exists \bar x_i\psi_i(\bar x_i)$.
Let $\bar b\in \B$ so that $(\B,\bar b)\models \psi_i(\bar b)$.
The structure $(\B,\bar b)$ satisfies the $\Pinf{\beta+1}$ formula $\psi_i(\bar x)$.
The theory $\psi_i(\bar x)$ is $\Sinf{\beta}$-small because for each of its models $(\C,\bar{c})$, $\C\models T$.
By \cref{lem:havingPimodels}, there is a model $(\A,\bar a)$ with $(\A,\bar a)\models \psi_i(\bar a)$ and $SR(\A,\bar a)\leq \beta$.
If we forget the names of the constants for $\bar a$, we obtain a structure $\A$ with $\A\models T$ and $SR_p(\A)\leq \beta$, as desired.
\end{proof}

By the definition of the different notion of Scott ranks, the above lemmas say that $\Sinf{\beta}$-small $\Pinf{\beta+1}$ classes have a model of Scott complexity at most $\Pinf{\beta+1}$, and $\Sinf{\beta}$-small $\Sinf{\beta+2}$ classes have a model of Scott complexity at most $\Sinf{\beta+2}$.
The following example shows that the analogous statement for Scott complexity $\dSinf{\beta+1}$ fails.
\begin{example}\label{ex:cutModels}
  Let $E_3$ be the well-known Ehrenfeucht theory with 3 models in the vocabulary $(<, (c_i)_{i\in\nat})$ with axioms consisting of $\Pi^0_2$ sentences stating that $<$ is a strict dense linear order without endpoints and $c_i< c_j$ if and only if $i<j$. If $U$ is the $\Pinf{2}$ sentence stating that the $c_i$ are unbounded in $<$, then $E_3\land \neg U$ has complexity $\dSinf{2}$ and no model of Scott complexity at most $\dSinf{2}$.
\end{example}
\begin{proof}
The theory $E_3$ has exactly three models:
\begin{enumerate}
	\item $\+U$ in which the $c_i$ are unbounded
	\item $\+L$ in which the $c_i$ are bounded and have a limit realized in $\+L$
	\item $\+N$ in which the $c_i$ are bounded but have no limit realized in $\+N$.
\end{enumerate}
We begin by calculating the Scott complexities of $\+U$, $\+L$, and $\+N$.
The model $\+U$ is distinguished from $\+L$ and $\+N$ by the $\Pinf{2}$ sentence $U$ stating that the $c_i$ are cofinal. Thus $E_3\cap U$ is a $\Pinf{2}$ Scott sentence for $U$. As $\+U$ is infinite, it cannot have a $\Sinf{2}$ Scott sentence and thus the Scott complexity of $\+U$ is $\Pinf{2}$.

To determine the Scott complexity of $\+N$ consider an element $x\in \+N$ to the right of every $c_i$.
We show that $x$ is 1-free.
Let the $\forall$-player play the tuple $x,\bar y,\bar z,\bar w$ where $\bar y $ is below some constant $c_m$, $\bar z$ is to the left of $x$ but above all $c_i$, $\bar w$ is to the right of $x$ and $|\bar y,\bar z, \bar w|=n$.
Let $x',\bar{z}'$ be a tuple in the order of $x,\bar{z}$ between $c_{\max(m,n)}$ and $c_{\max(m,n)+1}$. 
In response, the $\exists$-player will play $x',\bar y,\bar z',\bar w$.
It is straightforward to confirm that this is a winning play as $x',\bar y,\bar z',\bar w$ was selected to agree with $x,\bar y,\bar z,\bar w$ on the first $n$ formulas, so $x$ is 1-free as desired.
This means that $SR(\+N)\geq 2$.
The same argument works over any finite set of parameters and thus, in particular, $SR_p(\+N)\geq 2$. 
Furthermore, $\+N$ is distinguished from $\+U$ and $\+L$ by the sentence 
\[N:=\exists y \bigwwedge_i y>c_i \land \forall z \big( \bigwwedge_j z<c_i \rightarrow (\exists w<z)  \bigwwedge_i w>c_i\big).\]
Hence, $E_3\land N$ is a $\Pinf{3}$ Scott sentence for $\+N$, and so the Scott complexity of $\+N$ is $\Pinf{3}$.

Now, to calculate the Scott complexity of $\+L$, we know that $E_3\land \lnot U\land \lnot N$ is a satisfiable $\Sinf{3}$ sentence that is $\Sinf{1}$-small.
Therefore, there is a model (which must be $\+L$) that has Scott complexity at most $\Sinf{3}$.
Say that the Scott complexity of $\+L$ was at most $\dSinf{2}$.
Then $E_3\land \lnot U\land \lnot L$ would be a satisfiable $\Sinf{3}$ sentence that is $\Sinf{1}$-small, so it should have a model of Scott complexity at most $\Sinf{3}$.
That said, its only model is $\+N$, which has a Scott complexity of exactly $\Pinf{3}$, a contradiction.
Therefore, $\+L$ has Scott complexity exactly $\Sinf{3}$.

At last, consider the theory $E_3\land\lnot U$.
It is $\dSinf{2}$ and has exactly two models, $\+N$ and $\+L$.
These models have Scott complexity $\Pinf{3}$ and $\Sinf{3}$ respectively, demonstrating the claim.
\end{proof}

We further note that there are limitations to these claims near limit levels.

\begin{proposition}
There is a $\Pinf{\omega}$ theory that is $\Sinf{n}$-small for every $n$ that has no model of Scott complexity at most $\Pinf{\omega}$.
There is a $\Sinf{\omega+1}$ theory that is $\Sinf{n}$-small for every $n$ that has no model of Scott complexity at most $\Sinf{\omega+1}$.
\end{proposition}

\begin{proof}
Consider any structure $\A$ with Scott complexity above $\Pinf{\omega}$.
The $\omega$-type of $\A$ is described by a $\Pinf{\omega}$ formula $T$ (see~\cite{CGHT}, for example).
Furthermore, any $\B$ with $\A\equiv_\omega\B$ must realize the same $n$-types as $\A$ for every $n$, so $T$ is $\Sinf{n}$-small for every $n$.
Lastly, there is no structure modeling $T$ of Scott complexity at most $\Pinf{\omega}$, as if there were such a $\C$, $\C\equiv_\omega \A$ would mean $\C\cong\A,$ a contradiction to the assumption.

To see the second claim, we consider a specific structure in the above scheme.
This will, in fact, give us a $\Pinf{\omega}$ theory that is $\Sinf{n}$-small for every $n$ that has no model of Scott complexity  $\Sinf{\omega+1}$.
Let $\A$ be a prime model of a non-standard completion of Peano Arithmetic.
Any $\B$ with $\B\equiv_\omega \A$ will also be a model of this non-standard completion of Peano Arithmetic.
By~\cite{GLRS}, no such structure has Scott complexity at most $\Sinf{\omega+1}$.
\end{proof}

\subsection{Finding exact Scott complexities}\label{sec:exactSC}

The previous subsection was dedicated to constructing structures with given Scott analytic bounds.
This subsection provides a more exact analysis.
In particular, we aim to give conditions for the existence of structures that achieve exact Scott complexity invariants.
We begin with a counting condition that demonstrates the existence of $\Pi$ Scott sentences.

\begin{lemma}\label{lem:smallnessandgreaterscimpliessc}
If $T\in \Pinf{\beta+1}$ has $vo(T)>\beta$, then there is $\B\models T$ with Scott complexity $\Pinf{\beta+1}$.\end{lemma}
\begin{proof}
  Let $S_\beta=\{ \+S: \+S\models T \land SC(\+S)\leq \dSinf{\beta}\}$. This set is countable as $T$ is $\Sinf{\beta}$-small. It is axiomatized by a $\Sinf{\beta+1}$ formula given by the disjunction of the Scott sentences for each $S\in S_\beta$. Note that the intersection of its complement with models of $T$ is non-empty, as $T$ must have a model of Scott rank greater than or equal to $\beta$. So by \cref{lem:havingPimodels}, there is $\B\models T$, with $SR(\B)\leq \beta$ and Scott complexity greater than $\dSinf{\beta}$ (i.e., $\B\not\in S_\beta$). Thus, $\B$ has Scott complexity $\Pinf{\beta+1}$.
\end{proof}
Iterating \cref{lem:smallnessandgreaterscimpliessc} lets us prove \cref{thm:ssptrivial}.
\ssptrivial* 
\begin{proof}
  That $T$ has a model of Scott rank $\alpha$ follows from \cref{lem:smallnessandgreaterscimpliessc}. Fix $\beta\in (\alpha,vo(T))$. Let $S=\{ \+A\models T: SR(\+A)<\beta\}$. As $S$ has at most countably many isomorphism types, $S$ is definable by the $\Pinf{\beta+1}$ formula $\psi=T\land \bigwwedge_{\+A\in S/{\cong}} \neg\phi_\+A$ where $\phi_{\+A}$ is a $\Pinf{\beta}$ Scott sentence for $\+A$. Hence, $\psi$, and thus also $T$, have a model of Scott rank $\beta$ by \cref{lem:smallnessandgreaterscimpliessc}.
\end{proof}
We can push this argument further to guarantee the existence of multiple models if we sacrifice some control over the exact Scott invariants of those models.
This involves carefully chaining multiple applications of \cref{lem:havingPimodels} and \cref{lem:havingSigmamodels}.

\begin{lemma}
If $T\in \Pinf{\beta+1}$ has $vo(T)>\beta+n$ and has at least $2n$ models, then there are at least $2n$ models of Scott sentence complexity at most $\Pinf{\beta+n+1}$ and Scott rank at least $\beta$.
\end{lemma}
\begin{proof}
  Start with the structure $\+B$ with Scott complexity $\Pinf{\beta+1}$ from \cref{lem:smallnessandgreaterscimpliessc}. Let $\psi$ be a Scott sentence for $\+B$ and consider $\neg S_\beta\land T\land \neg \psi$ where $S_\beta$ is the set from the proof of the above lemma. This is a consistent $\Sinf{\beta+2}$ formula as $T$ has models of Scott complexity at least $\Pinf{\beta+1}$. \cref{lem:havingSigmamodels} gives a model of $\neg S_\beta\land T\land \neg \psi$ with Scott complexity less than $\Sinf{\beta+2}$. Call this model $\+A$ and its Scott sentence $\theta$. We then have that $\neg S_\beta\land \neg \psi\land \neg \theta$ is a consistent $\Pinf{\beta+2}$ formula. By \cref{lem:havingPimodels} this sentence has a model with Scott complexity $\Pinf{\beta+2}$. Iterating this process gives models $\B_i$ of complexity at most $\Pinf{\beta+i}$ with $\B_1=\B$ and $\A_i$ of complexity at most $\Sinf{\beta+i+1}$ where $\A_1=\A$. 
  \end{proof}

We now examine the analogous proposition for $\Sigma$ formulas.
This does not work out as nicely as the general situation with $\Pi$ formulas, but there are still useful conclusions to be drawn.
    
\begin{proposition}\label{prop:sigmasmallnessandgreaterscimpliessc}
If $T$ is $\Sinf{\beta+2}$ and has $vo(T)>\beta+2$, then either there exists a structure with Scott complexity exactly $\Sinf{\beta+2}$, or there exists infinitely many structures with parameterized Scott rank $\beta$.
\end{proposition}

\begin{proof}
Let $R_{\beta}=\{\A : \A\models T \land SR_p(\A)<\beta\}$. This set is countable as $T$ is $\Sinf{\beta}$-small and axiomatized by a $\Sinf{\beta+1}$ formula given by the disjunction of all $\Sinf{\beta+1}$ Scott sentences for models in $R_\beta$. The intersection of the complement of $R_\beta$ with models of $T$ is non-empty because $T$ has a model of Scott rank at least $\beta+2$. So by \cref{lem:havingSigmamodels}, there is $\B_0\models T$, with $SR_p(\B_0)\leq\beta$ and $\B_0\not\in R_\beta$. Note that this is the same as saying $SR_p(\B_0)=\beta$.  Consider the case where $\B_0,\cdots,\B_i$ are models of $T$ with parameterized Scott rank $\beta$ that do not have Scott complexity exactly $\Sinf{\beta+2}$. In this case, we note that the complement of $R_\beta$ along with the condition that a model is not isomorphic to any of $\B_0,\dots,\B_i$ is $\Sinf{\beta+2}$. So by \cref{lem:havingSigmamodels}, there is $\B_{i+1}\models T$, with $SR_p(\B_{i+1})=\beta$. Continuing in this fashion, there is either some $i$ with $SC(\B_{i})=\Sinf{\beta+2}$, or the process continues giving infinitely many models with parameterized Scott rank $\beta$.
\end{proof}

Note that the above proposition is equally provable with the weaker assumption that $T$ has some model of parameterized Scott rank at least $\beta$ and is $\Sinf{\beta}$ small.
This makes the assumption more analogous to the unparameterized version.
The asymmetry comes from the use of the unparamaterized Scott rank in the definition of the Vaught ordinal. 

\begin{corollary}
If $T$ is $\Pinf{\beta+1}$ and $vo(T)>\gamma+2$, then there are at least two structures of parameterized Scott rank $\alpha$ for every $\alpha\in[\beta,\gamma)$.
\end{corollary}

\begin{proof}
Fix such an $\alpha$. By \cref{prop:sigmasmallnessandgreaterscimpliessc}, we may assume that there is a structure of $T$ with Scott complexity $\Sinf{\beta+2}$. By \cref{thm:ssptrivial}, there is also a structure of $T$ with Scott complexity $\Pinf{\beta+1}$.
\end{proof}

The above corollary yields a strengthening of Sack's theorem \cite{Sac83}, who showed that a $\Pinf{2}$ counterexample to Vaught's conjecture must have two models of every $\Sigma_2^1$ admissible Scott rank.

\twomodelsvc

We will see in the subsequent section that the above corollary fails for the unparameterized Scott rank.
Indeed, it fails at $\omega_1$ many ranks.

\section{Labeling Back-and-Forth Classes}
This section focuses on conclusions that can be drawn about the back-and-forth relations below the Vaught ordinal.
We begin with some general theory regarding the back-and-forth relations.
Combining this with the work of the previous section, we see four applications.
The first is a contrast to \cref{cor:2ModelsVC}, showing that failure of Vaught's conjecture implies the existence of a counterexample $T$ with many ordinals $\alpha\in\omega_1$ such that there is only one $\M\models T$ with $SR(\M)=\alpha$.
We then provide a sharpening of \cref{prop:sigmasmallnessandgreaterscimpliessc} from the previous section, exploring the necessity of models whose parameterized Scott ranks differ from the unparamaterized Scott rank.
The third subsection gives a procedure to understand the classification of $SR(\M)\leq\alpha$ models in the $\Sinf{\alpha}$-small setting, focusing on particular orderings as key examples.
Finally, we examine non-splitting types and Boolean algebras to answer a question of Harris and Montalb\'an.

Our first proposition has nothing to do with Vaught ordinals or other counting assumptions.
Instead, it is a generally useful fact regarding the behavior of the back-and-forth relations.

\uniqueminSR

\begin{proof}
    We describe a strategy to the $\mathcal{A}\geq_{\alpha+1}\mathcal{B}$ game.
    Let the $\forall$-player select $\bar{a}\in\mathcal{A}$ on their first turn of the game.
    Let $\exists\bar{y}\theta(\bar{x},\bar{y})$ with $\theta\in\Pinf{\beta}$ for $\beta<\alpha$ define the automorphism orbit of $\bar{a}$.
    Find $\bar{c}\in\mathcal{A}$ such that $\mathcal{A}\models\theta(\bar{a},\bar{c})$.
    If the $\forall$-player plays $\bar{a},\bar{c}$ along with the ordinal $\beta$ in the $\mathcal{A}\equiv_{\alpha}\mathcal{B}$ game, the $\exists$-player has some winning response $\bar{b},\bar{d}$.
    We claim that $\bar{b}$ is a winning response to $\bar{a}$ for the $\exists$-player in the $\mathcal{A}\geq_{\alpha+1}\mathcal{B}$ game.
    In particular, we must show that $(\mathcal{A},\bar{a})\leq_{\alpha}(\mathcal{B},\bar{b}).$
    Equivalently, we show that every $\Sinf{\alpha}$ statement $\exists \bar{z} \rho(\bar{x},\bar{z})$ with $\rho\in\Pinf{\gamma}$ for some $\gamma$ with $\beta\leq\gamma<\alpha$ that is true of $\bar{b}$ is also true of $\bar{a}$.
    Consider $\bar{p}$ such that $\mathcal{B}\models\rho(\bar{b},\bar{p}).$
    If the $\forall$-player plays $\bar{b},\bar{d},\bar{p}$ along with the ordinal $\gamma$ in the $\mathcal{A}\equiv_{\alpha}\mathcal{B}$ game, the $\exists$-player has some winning response $\bar{c},\bar{f},\bar{q}$.
    In particular, $\mathcal{A}\models \rho(\bar{c},\bar{q})\land \theta(\bar{c},\bar{f})$.
    This means that, $\mathcal{A}\models \exists\bar{z}\rho(\bar{c},\bar{z})\land \exists\bar{y}\theta(\bar{c},\bar{y})$.
    By the definition of $\theta,$ this means that $\bar{c}\cong\bar{a}$ in $\mathcal{A}$.
    Furthermore, we observe that necessarily $\mathcal{A}\models \exists\bar{z}\rho(\bar{a},\bar{z})$ as well.
    In other words, every  $\Sinf{\alpha}$ statement $\exists \bar{z} \rho(\bar{x},\bar{z})$ with $\rho\in\Pinf\gamma$ that is true of $\bar{b}$ is also true of $\bar{a}$, so $(\mathcal{A},\bar{a})\leq_{\alpha+1}(\mathcal{B},\bar{b}),$ as desired.
\end{proof}

\begin{corollary}\label{cor:onlyOneSRalphaPerClass}
        Let $SR(\mathcal{A})=SR(\mathcal{B})=\alpha$. If $\mathcal{A}\equiv_{\alpha}\mathcal{B}$ then $\mathcal{A}\cong\mathcal{B}$.
\end{corollary}

\begin{proof}
    By the above Proposition, $\mathcal{A}\geq_{\alpha+1}\mathcal{B}$.
    Because $SR(\mathcal{B})=\alpha$, $\mathcal{B}$ has a $\Pinf{\alpha+1}$ Scott sentence.
    Therefore, $\mathcal{A}\cong\mathcal{B}$ as desired.
\end{proof}

The corollary above demonstrates that there cannot be more than one simple structure in a given $\equiv_{\alpha}$ class.
This allows us to define the unique $SR(\mathcal{B})=\alpha$ structure in a given $\equiv_{\alpha}$ class, if it exists.

\begin{definition}
For each $\beta\in\omega_1$, a given $\equiv_\beta$-back-and-forth class is \define{labeled} if there is a structure of Scott rank at most $\beta$ in that class.
In such a case, the \define{label} of the class is given by the unique structure of Scott rank at most $\beta$ in that class.
\end{definition}

Labels give us a way to understand the $\Sinf{\beta}$ and $\Pinf{\beta}$ facts in a given $\equiv_\beta$-back-and-forth class in terms of a concrete structure.
Labels always exist under suitable smallness conditions.\footnote{We are grateful to Serafim Batzoglou for pointing out that the assumption that $T$ is $\Pinf{\beta+1}$ is crucial in \cref{lem:smallEasybnfDescription}. The missing assumption was discovered by analyzing a prior version of this paper using GPT 5.6 Sol Pro.}

\begin{lemma}\label{lem:smallEasybnfDescription}
  If $T$ is $\Pinf{\beta+1}$  and $\Sinf{\beta}$-small, then for fixed $\+M\models T$, the set
\[E(\+M,\beta):=\{\+N : \+N\models T \land \+N\equiv_\beta \+M\}\]
is $\Pinf{\beta+1}$.
\end{lemma}

\begin{proof}
Because $T$ is $\Sinf{\beta}$ small, there are only countably many $\equiv_\beta$-classes among the models of $T$.
Enumerate these classes $\{C_i\}_{i\in\omega}$ where $C_0= E(\+M,\beta)$.
Say that $C_i$ for $i>0$ differs from $C_0$ on the $\Pinf{\beta}$ or $\Sinf{\beta}$ formula $\psi_i$.
It is straightforward to see that
\[\+N\in E(\+M,\beta) \iff \+N\models T \land \+N\models\bigwwedge_{i>0} \lnot\psi_i,\]
giving the desired $\Pinf{\beta+1}$ description.
\end{proof}

\begin{corollary}\label{cor:allClassesLabeled}
  If $T$ is $\Pinf{\beta+1}$ and $\Sinf{\beta}$-small, then every $\beta$-back-and-forth class of $T$ is labeled.
\end{corollary}

\begin{proof}
This follows directly from \cref{lem:smallEasybnfDescription} and \cref{lem:havingPimodels}.
\end{proof}

We can use labels to further understand the notion of $\alpha$-universality defined in \cite{GR24}.
We recall this definition now.
\begin{definition}
Recall that $S$ is an \define{$\alpha$-universal class} for $K$ if for all $\+A\in K$ there is $\+B\in S$ such that $\+B\geq_{\alpha}\+A$.
\end{definition}
The notion of $\alpha$-universality played a key role in the classification of simple linear orderings from \cite{GR24}.

\begin{corollary}\label{cor:SRClassificationFromSmall}
  If $T$ is $\Pinf{\beta+1}$ and $\Sinf{\beta}$ small, then $T_{\leq\beta} = \{\+M\models T: SR(\+M)\leq \beta\}$ is the set of labels on the $\equiv_\beta$ classes, $L_\beta$. In other words, $T_{\leq\beta}$ is the minimal $(\beta+1)$-universal class for models of $T$.
\end{corollary}

\begin{proof}
By definition, $L_\beta\subseteq T_{\leq\beta}$, as labels of $\equiv_\beta$ classes must be Scott rank at most $\beta$.
Consider $\+M\in T_{\leq\beta}$.
$\+M$ resides in some $\beta$-back-and-forth class, which is labeled by \cref{cor:allClassesLabeled}.
This means there is $\+N\in L_\beta$ such that $\+M\leq_{\beta+1} \+N$.
Because $\+M\in T_{\leq\beta}$, this gives that $\+M\cong \+N$, so $\+M\in L_\beta$.

\cref{cor:allClassesLabeled} states that every model of $T$ is $(\beta+1)$-below a model of $T_{\leq\beta}$.
In other words, $T_{\leq\beta}$ is a $(\beta+1)$-universal class for models of $T$.
Furthermore, any such $(\beta+1)$-universal class must contain all models in $T_{\leq\beta}$ as these models are only $(\beta+1)$ below themselves.
Therefore, $T_{\leq\beta}$ is the minimum $(\beta+1)$-universal class.
\end{proof}

\subsection{Unparameterized ranks and Vaught's conjecture}
We now have established the needed tools to give an unparamaterized contrast to \cref{cor:2ModelsVC}.
The key is to analyze a minimal counterexample to Vaught's conjecture. 
We define this notion before moving on.

\begin{definition}[\cite{HM77}]
$T$ is a \define{minimal counterexample} to Vaught's conjecture if for any $\varphi\in L_{\omega_1,\omega}$ either $T\land\varphi$ or $T\land\lnot\varphi$ has only countably many models.
\end{definition}

In~\cite{HM77}, Harnik and Makkai showed that if there is any counterexample to Vaught's conjecture, then there must be a minimal one.

\begin{theorem}\label{thm:minimalcealeph1}
Given $\varphi$ a minimal counterexample to Vaught's conjecture, there is a function $G:\omega_1\to\omega_1$ with $G(\alpha)\geq\alpha$ such that for all $\alpha\in\omega_1$, $\varphi$ has exactly one model of unparameterized rank $G(\alpha)$.
\end{theorem}

\begin{proof}
Fix $\alpha\in\omega_1$.
List the $\equiv_\alpha$-classes of $\varphi$ as $C_{1,\alpha},C_{2,\alpha},\cdots$.
Note that there are countably many, as $\varphi$ is a counterexample to Vaught's conjecture.
Say that $C_{i,\alpha}$ and $C_{j,\alpha}$ have uncountably many models.
By \cref{lem:smallEasybnfDescription}, $C_{i,\alpha}$ and $C_{j,\alpha}$ are disjoint, axiomatizable subclasses of $\varphi$.
Because $\varphi$ is a minimal counterexample, this is a contradiction.
Therefore, there is a unique $i$ such that $C_{i,\alpha}$ is uncountable.
For each $\alpha$ we write $C_\alpha$ to denote the unique uncountable $\equiv_\alpha$-class of $\varphi$.
Note that if $\lambda$ is a limit ordinal, then $C_\lambda = \bigcap_{\beta<\lambda} C_\beta$, as each $\equiv_\lambda$-class is an intersection of $\equiv_\beta$ classes and the $C_\beta$ are the only ones of the appropriate size.

Define the following function $f:\omega_1\to\omega_1$
\[f(\alpha)=\min\{ \beta\in\omega_1 | (\forall \+M\not\in C_\alpha) SR(\mathcal{M})<\beta\}. \]
This function is well-defined because there are only countably many $\mathcal{M}$ with $\mathcal{M}\not\in C_\alpha$.
Consider a fixed point $f(\alpha)=\alpha$.
In this case, every model $\mathcal{N}$ outside of $C_\alpha$ has Scott rank below $\alpha$.
This means that every $\mathcal{M}$ with $SR(\mathcal{M})\geq\alpha$ has $\mathcal{M}\in C_\alpha$.
However, by \cref{cor:onlyOneSRalphaPerClass}, $C_\alpha$ has at most one model $\mathcal{K}$ with $SR(\mathcal{K})\leq\alpha$.
Because $C_\alpha$ has more than one model, in fact, $SR(\mathcal{K})=\alpha$.
As there are no such models outside of $C_\alpha$, there is only one model of Scott rank $\alpha$ and so $G(\alpha)=\alpha$ has the required properties.
Furthermore, if for some $n\in\omega$, $f^n(\alpha)=f^{n+1}(\alpha)$, then setting $G(\alpha)=f^n(\alpha)$ has the required properties as $f^n(\alpha)$ is an $f$ fixed point.

Consider instead the case where, for all $n$, $f^n(\alpha)<f^{n+1}(\alpha)$.
In this case, we claim that $G(\alpha)=\sup_{n\in\omega} f^n(\alpha)$ has the required properties.
This comes down to showing that every model outside of $C_{G(\alpha)}$ has Scott rank below $G(\alpha)$; from there, the argument above demonstrates the claim.
Let $\mathcal{N}\not\in C_{G(\alpha)}$.
Because $G(\alpha)$ is a limit ordinal, this means that there is some $\beta<G(\alpha)$ where $\mathcal{N}\not\in C_\beta$.
More saliently, there is an $n\in\omega$ such that $\mathcal{N}\not\in C_{f^n(\alpha)}$.
By the definition of $f$, this gives that $SR(\mathcal{N})<f^{n+1}(\alpha)<G(\alpha)$, as desired.
\end{proof}

\subsection{Models for which parameters matter}
We establish a sharpening of \cref{prop:sigmasmallnessandgreaterscimpliessc}.
This stronger result allows us to conclude that there is no avoiding the existence of structures whose parameterized and nonparameterized Scott ranks differ in the context of countability.
For example, counterexamples to Vaught's conjecture must have such models.

\begin{proposition}\label{prop:parametersMatter}
If $T$ is $\Sinf{\beta+2}$ and has $vo(T)>\beta+2$, then either there exists a structure with Scott complexity exactly $\Sinf{\beta+2}$ or there exists infinitely many structures with Scott complexity exactly $\dSinf{\beta+1}$.
\end{proposition}

\begin{proof}
Fix a model $\B\models T$ with Scott complexity at least $\Sinf{\beta+2}$.
By \cref{lem:smallEasybnfDescription} $E(\+B,\beta):=\{\+A : \+A\models T \land \+A\equiv_\beta \+B\}$ is $\Pinf{\beta+1}$ definable.
In particular, $T\land E(\B,\beta)$ is a $\Sinf{\beta+2}$ theory that still has a model of Scott complexity at least $\Sinf{\beta+2}$ and is $\Sinf{\beta}$ small (indeed $vo(T\land E(\B,\beta))>\beta+2$).
By applying \cref{prop:sigmasmallnessandgreaterscimpliessc} to $T\land E(\B,\beta)$, we obtain that there is either a model of Scott complexity $\Sinf{\beta+2}$ that is $\equiv_\beta$ equivalent to $\B$ or there are infinitely many such models with $SR_p(B_i)=\beta$.
In the former case, we are done at once.
In the latter case, by \cref{cor:onlyOneSRalphaPerClass}, at most one of these structures has Scott complexity $\Pinf{\beta+1}$.
In the absence of any models of Scott complexity $\Sinf{\beta+2}$, this means that the remaining infinitely many must have Scott complexity $\dSinf{\beta+1}$, as desired.
\end{proof}

Note that the above proposition works equally well with the assumption that $T$ is $\Sinf{\beta}$-small and has a model of Scott complexity at least $\Sinf{\beta+2}$, just like \cref{prop:sigmasmallnessandgreaterscimpliessc} before it.
The claim mentioned about counterexamples to Vaught's conjecture follows at once.

\begin{corollary}
  If $T$ is a $\Pinf{2}$ counterexample to Vaught's conjecture, then for every $\beta\in\omega_1$ there is a structure $\+M$ such that $SR_p(\+M)=\beta<SR(\+M)$.
\end{corollary}

An unexpected consequence of \cref{prop:parametersMatter} is an entirely abstract resolution of a question posed by Alvir, Csima, and MacLean \cite[Question 5.2]{ACM24}.
In particular, we show that there is an Abelian p-group with a non-$\Pi$ Scott complexity.

\abelianpsigma

\begin{proof}
Khisamiev showed that the theory of Abelian p-groups is $\Sinf1-$small~\cite{Khi04}.
It is well known that Abelian p-groups have models of arbitrarily high Scott rank (see e.g.~\cite{ACM24}).
The result follows at once from \cref{prop:parametersMatter}.
\end{proof}

This is, in some sense, an unsatisfactory answer to the question.
Ideally, one would like to have an actual construction of such a group, as it is likely the concrete group could be modified to create non-$\Pi$ Scott complexity at higher levels as well.
While it is now clear that a search for such a non-$\Pi$ Scott complexity construction would not be in vain, we do not attempt the search here.
To be more explicit, we ask the following question, updating Alvir, Csima, and MacLean's question with our new knowledge.

\begin{question}
  Explicitly construct an Abelian p-group of parameterized Scott rank $1$ and Scott rank greater than $1$. Which non-$\Pi$ Scott complexities are attainable for Abelian p-groups?
\end{question}

Given the work of Harris and Montalbán \cite{HM} who calculated that the Vaught ordinal of Boolean algebras is $\omega$, we may analogously conclude that there are $\Sinf{n+2}$ Scott complexity Boolean algebras or infinitely many $\dSinf{n+1}$ Boolean algebras for each $n\in\omega$.

We note that both of the possibilities in \cref{prop:parametersMatter} are necessary, and that this is apparent at even the lowest levels of the hierarchy.
\begin{proposition}
There is a $\Sinf1$-small theory with a Scott complexity $\Sinf3$ model but no Scott complexity $\dSinf{2}$ models.
There is a $\Sinf1$-small theory with no Scott complexity $\Sinf3$ model but infinitely many Scott complexity $\dSinf{2}$ models.
\end{proposition}

\begin{proof}
  The first theory is given by \cref{ex:cutModels}. An example for the second theory is the theory of linear orderings.
It is shown in \cite{GR24} that there is no Scott complexity $\Sinf3$ linear ordering, yet there are countably many Scott complexity $\dSinf2$ linear orderings given by those with a non-zero, finite number of successivities.
The classic fact that linear orderings are $\Sinf1$-small was initially proved by Richter~\cite{richter1981}, albeit not isolated as such.
\end{proof}

\subsection{Classifying models of a given Scott rank}

We now examine an example that previously seemed like a fortunate coincidence, but can now be understood in terms of the larger pattern explained above.
It was shown by Knight \cite{Kni86} that linear orderings are $\Sinf{2}$ small.
Montalban took care to describe the $\equiv_2$ classes explicitly \cite[Section 4]{McountingBF}.
A number of years later, Gonzalez and Rossegger showed the following.
\begin{theorem}[\cite{GR24} Theorem 2.7]
Let
\[K=\{\omega+k, k+\omega^*, k+\zeta+k':k,k'\in\mathbb N\}\cup\{\omega+\omega^*\}\cup\{k:k\in\mathbb N\}\cup\]
\[\{(\sum_{i\in k} n_i + m_i\cdot\eta )+n_{k+1}:n_i,m_i\in\mathbb N,
n_i>m_{i-1},m_{i+1}.\}\]
Then,
\begin{enumerate}
	\item $K$ is the set of all linear orderings with Scott rank at most $2$.
	\item $K$ is a minimal $3$-universal class.
	\item There is exactly one element of $K$ in each $\equiv_2$ class of linear orderings.
\end{enumerate}
\end{theorem}
We can now understand the above theorem as following directly from Knight~\cite{Kni86} and Montalbán~\cite[Section 4]{McountingBF}, even though the original proof does not reckon with these results at all.
In particular, Knight's result that linear orderings are $\Sinf{2}$-small, along with \cref{cor:SRClassificationFromSmall}, immediately gives that $LO_{\leq2}$ is a minimal $3$-universal class (i.e. $(1)$ implies $(2)$ above).
Furthermore, \cref{cor:onlyOneSRalphaPerClass} yields that there can only be one element of $LO_{\leq2}$ in each $\equiv_2$ class (i.e. $(1)$ implies $(3)$ above).
Lastly, to get an explicit description of $K$ (i.e., (1) above), it suffices by \cref{cor:SRClassificationFromSmall} to look at each $\equiv_2$ class and find the simplest structure in it; this can be done by careful examination of Montalban's explicit $\equiv_2$ classification for linear orderings.
In total, this approach gives a general outline for moving from results that count types and back-and-forth equivalence classes to results that classify the simplest structures of a theory.

This idea also provides insights into theories where the Scott analysis is not as completely understood.
For example, we may examine the class of order-theoretic trees $Tr$.
$Tr$ is more formally defined as the set of partial orderings $P$ such that the set of elements less than any given $x\in P$ is linearly ordered.
It was shown by Richter \cite{Ric81} that $Tr$ is $\Sinf{1}$ small.
\cref{cor:SRClassificationFromSmall} gives that there are countably many trees with Scott rank 1, one for each $\equiv_1$ class.
This can be understood in terms of finite trees in the following, more structural manner.

\begin{definition}
Let $FinTr$ denote the finite models of $Tr$.
Given a finite tuple $\bar{U}\in FinTr^{<\omega}$, let $C(\bar{U})\subseteq FinTr$ be the elements of $FinTr$ that do not take any embeddings from any elements of $\bar{U}$.
\end{definition}

\begin{theorem}
There is an mapping $\Theta: Tr_{\leq 1}\to FinTr^{<\omega}$ such that
\begin{enumerate}
	\item $T\cong S \iff \Theta(T)=\Theta(S)$,
	\item $Age(T)= C(\Theta(T))$,
	\item $\bar{U}\in Im(\Theta)$ if and only if $C(\bar{U})$ has the Joint Embedding Property and $\bar{U}$ is an anti-chain.
\end{enumerate}
\end{theorem}

\begin{proof}
We begin by defining $\Theta$.
Fix $T\models Tr$ and consider $X(T)$, the complement of $Age(T)$ in $FinTr$.
By Kruskal's tree theorem \cite{Kru60}, $FinTr$ is well quasi-ordered by embeddability.
In particular, $X(T)$ is empty or has finitely many embeddability-minimal elements.
If $X(T)$ is empty, we let $\Theta(T)=\emptyset$ (note that $C(\emptyset)=FinTr$).
If $X(T)$ has finitely many embeddability-minimal elements, we let $\Theta(T)$ be the collection of those elements.

We now prove the claimed properties for $\Theta$.
We begin with (2).
In the case that $\Theta(T)=\emptyset$, we have that $C(\Theta(T))=FinTr$.
We also have that $X(T)$ is empty, or, what is the same, $Age(T)=FinTr$.
Now say that $\Theta(T)$ is not empty.
Note that $X(T)$ is upwards embeddability-closed.
In particular, it is the upward closure of its minimal elements, $\Theta(T)$.
This means that $Age(T)$ is the complement of the embeddability cones above $\Theta(T)$.
By defintion, this gives that $Age(T)= C(\Theta(T))$.

We now show (1).
Because $Age(T)$ is isomorphism invariant, it follows at once that $\Theta(T)$ is too.
Now consider $T$ and $S$ with $\Theta(T)=\Theta(S)$.
By (2), $Age(T)$ is determined by $\Theta(T)$, so we have that $Age(T)=Age(S)$.
We now demonstrate that $Age(T)=Age(S)$ implies that $T\cong S$.
More specifically, we show that $T\equiv_1S$ and note that \cref{cor:onlyOneSRalphaPerClass} yields $T\cong S$.
Say that the $\forall$-player plays $\bar{p}\in T$ on their move.
$\bar{p}$ has an induced ordering on it, making it an element of $Age(T)$.
Because $Age(T)=Age(S)$, we obtain that there is a $\bar{q}\in S$ with the same induced ordering.
The $\exists$-player may play this $\bar{q}$ in response, winning the game.
A symmetrical strategy gives that $T\equiv_1S$, and so $T\cong S$ as desired.

We lastly show (3).
If $\bar{U}\in Im(\Theta)$, then, by (2), $C(\bar{U})$ is the age of a structure.
Therefore, it must have JEP.
$\bar{U}$ is an anti-chain because it is the set of minimal elements of $X(T)$ for some structure $T$.
Say that $C(\bar{U})$ has JEP and $\bar{U}$ is an anti-chain.
If $F\in C(\bar{U})$ and $G\subseteq F$, then certainly no elements of $\bar{U}$ embed into $G$.
Therefore, $G\in C(\bar{U})$, and so $C(\bar{U})$ has the Hereditary Property.
By Fraïssé's theorem \cite[Chapter 7.1]{Hod93}, there is some countable structure $\M$ with $Age(\M)=C(\bar{U})$.
Let $Age(\A)\neq Age(\M)$.
Then there is some $\bar{p}$ in a finite arrangement in one of $\A$ or $\M$ not witnessed in the other structure.
Playing this tuple gives a winning strategy for the $\forall$-player to demonstrate that $\A\not\equiv_1\M$.
Taken together, we obtain that every structure 1-equivalent to $\M$ has the same age.
By \cref{cor:SRClassificationFromSmall}, we have that there is some $\A$ with $SR(\A)=1$ and $\A\geq_2 \M$.
It follows that $Age(\A)=C(\bar{U})$, and so $\bar{U}=\Theta(\A)$.
Therefore, $\bar{U}\in Im(\Theta)$, as desired.
\end{proof}

A more explicit description of the Scott rank 1 trees can therefore be given by answering the following combinatorial question, which we leave open.
\begin{question}
For which finite anti-chains $\bar{U}\in FinTr^{<\omega}$ does $C(\bar{U})$ have the Joint Embedding Property?
\end{question}

\subsection{Non-splitting types}
In their investigation into Boolean algebras, Harris and Montalbán \cite{HM} described the collection of $n$-back-and-forth types for finite $n\in\omega$.
A direct consequence of their result is that they showed that the theory of Boolean algebras is $\Sinf{n}$ small for every $n$, as mentioned in the introduction of this paper.
A fundamental feature of their study was an interest in the ``descendant'' relation.
Formally speaking, for $m\geq n$, they defined an $m$-type to be a descendant of an $n$-type if the $n$-type is a consequence of the $m$-type.
We give an analogous definition here, including the transfinite case.
\begin{definition}
Given $\alpha,\beta\in\omega_1$ with $\alpha\geq\beta$, a theory $T\in L_{\omega_1,\omega}$, $\mathcal{A}\models T$ and  $\mathcal{B}\models T$, we say that $Th_{\Pinf{\alpha}}(\mathcal{A})$ is a \define{descendant} of $Th_{\Pinf{\beta}}(\mathcal{B})$ if $Th_{\Pinf{\beta}}(\mathcal{A}) = Th_{\Pinf{\beta}}(\mathcal{B})$.
\end{definition}

We also make the following definition, analogous to Harris and Montalbán.
\begin{definition}
A $\Pinf{\alpha}$ theory is \define{non-splitting} if it has exactly one $\Pinf{\alpha+1}$ theory as a descendant.
\end{definition}

Harris and Montalbán conjectured that, for Boolean algebras, non-splitting types at finite levels are exactly the isomorphism types, though they were unable to provide a general proof (instead only demonstrating the claim for $n=1,2,3,4$).
We show a more abstract result that follows immediately from our previous investigation, and confirm their full conjecture as a corollary.

\begin{proposition}\label{prop:nonsplittingTypes}
  Let $T$ in $\Pinf{\alpha+1}$ be a $\Sinf{\alpha}$-small theory and $\rho$ the $\Pinf{\alpha}$ type of $\mathcal{A}\models T$.
If $\rho$ is non-splitting, then for any $\mathcal{B}\models\rho\land T$, $\mathcal{B}\cong\mathcal{A}$.
\end{proposition}

\begin{proof}
By \cref{cor:allClassesLabeled}, $\rho$ is labeled, or, in other words, there is a $\mathcal{C}$ such that
\[\mathcal{C}\equiv_\alpha\mathcal{B}\equiv_\alpha\mathcal{A}\]
with $SR(\mathcal{C})\leq\alpha$.
Because $\rho$ is non-splitting, we obtain that 
\[\mathcal{C}\equiv_{\alpha+1}\mathcal{B}\equiv_{\alpha+1}\mathcal{A}.\]
Because $\mathcal{C}$ has a $\Pinf{\alpha+1}$ Scott sentence, we obtain that  
\[\mathcal{C}\cong\mathcal{B}\cong\mathcal{A},\]
as desired.
\end{proof}

As promised, we confirm the conjecture from \cite[Page 21]{HM}.

\begin{corollary}
Every non-splitting $\Pinf{n}$ type of Boolean algebras is an isomorphism type. 
\end{corollary}

\begin{proof}
  It was demonstrated in \cite{HM} that the theory of Boolean algebras is $\Sinf{n}$-small. Since the theory of Boolean algebras is $\Pinf{1}$, the result follows at once from \cref{prop:nonsplittingTypes}.
\end{proof}

\section{Ehrenfeucht theories}
We now turn our attention to Ehrenfeucht theories, or theories with only finitely many models.
The study of Ehrenfeucht theories is inextricably linked with the study of Vaught's conjecture (see \cite{PT25}).
Indeed, Vaught originally made his conjecture in \cite{Vau61}, where he excludes the possibility of a complete, first-order theory having exactly two models. 
Fortunately, our methods give us insights into the Scott analytic behavior of Ehrenfeucht theories alongside the previously stated insights into Vaught's conjecture.

Our goal is to explain which Scott complexities may arise as models of an Ehrenfeucht theory.
We formalize this using the Scott complexity counting functions introduced in \cite{GLRS}.

\begin{definition}
	Given a theory $T\in L_{\omega_1,\omega}$ and a complexity $\Gamma\in\{\Sinf{\alpha},\dSinf{\alpha},\Pinf{\alpha}\}_{\alpha\in\omega_1}$, the \define{Scott complexity counting function} $I$ is defined by
    \[ I(T,\Gamma) = |\{\mc{M}\models T : SSC(\mc{M})=\Gamma\}| .\]
\end{definition}

Understanding the behavior of the Scott complexity counting function gives the finest possible Scott analysis for a given theory.
We will call the full behavior of the Scott complexity counting function the \textbf{Scott complexity spectrum} of the theory.
We aim to understand the possible Scott complexity spectra for Ehrenfeucht theories, starting with the case where $T$ has exactly two models.
It is readily possible for an infinitary theory to have exactly two models (unlike the complete first-order case).
We state our results for $\Pinf{2}$-theories but want to emphasize that our setting is still fully general. Given any $\Pinf{\alpha}$-theory $T$ we can change the vocabulary using Morleyization to obtain a $\Pinf{2}$ theory whose models are precisely the expansions of models of $T$ (see \cite[Chapter II.5]{Part2}).

\begin{proposition}\label{prop:2ModelsClassified}
A $\Pinf{2}$ formula $T$ with exactly two models has one of the following two Scott complexity spectra:
\begin{align} &I(T,\Pinf{2})=2 \text{ and } I(T,\Gamma)=0 \text{ if $\Gamma\neq \Pinf{2}$};\\
 &I(T,\Pinf{2})=1,  I(T,\dSinf{2})=1, \text{ and } I(T,\Gamma)=0 \text{ if $\Gamma\not\in \{\Pinf{2},\dSinf{2}\}$.}\end{align}
Both possibilities are realized.
\end{proposition}

\begin{proof}
By \cref{lem:smallnessandgreaterscimpliessc} there must be at least one model $\A$ of complexity $\Pinf{2}$.
The other model $\B$ is axiomatized by saying $\B\not\cong\A\land\B\models T$.
This is a $\dSinf{2}$ condition.
This leaves only the two possibilities described in the proposition statement.

To obtain a $\Pinf{2}$ theory $T$ with two models, each of Scott complexity $\Pinf{2}$, simply consider two structures in the same vocabulary with $\Pinf{2}$ Scott complexity and take as $T$ the disjunction of their Scott sentences.

It remains to describe a $\Pinf{2}$ theory with two models, one of which has Scott complexity $\Pinf{2}$ and the other has Scott complexity $\dSinf{2}$.
Let $T$ be in the vocabulary of the binary relation $<$ stating that $<$ is a linear ordering.
Furthermore, $(\forall x<y) \exists z\ x<z<y\in T$ and $\forall x\exists y \ x<y\in T$, or in other words, any model of $T$ must be a dense linear ordering without last element.
It is straightforward to confirm that $\eta$ and $1+\eta$ are the only two models of $T$. 
At last, note that $1+\eta\leq_2 \eta$ (see~\cite[Lemma 2.4]{GR24}), so the structures must have the desired complexities.
\end{proof}

We have a similar characterization of theories with three models.
It will be convenient for us to establish some operations on theories to describe how to realize the given Scott complexity spectra in this more complicated case.

\begin{definition}
Given a theory $T$ over the vocabulary $L$, $T^+$ is a new theory defined over the vocabulary $L\cup\{R,U\}$ where $R$ is a fresh binary relation symbol and $U$ is a fresh unary relation symbol.
$T^+$ is the conjunction of 
\begin{enumerate}
	\item $\forall x\ \lnot U(x)\to (T\land \forall y,z \lnot R(y,z))$
	\item for any relation $S\in L$, $\exists x\  U(x) \to \forall \bar{z} \lnot S(\bar{z})$
	\item $\exists x\ U(x) \to \exists! y\ \lnot U(y)$ and for any constant $c\in L$, $y=c$
	\item for any function $f\in L$, $\exists x\  U(x) \to \forall \bar{z}\  \lnot U(f(\bar{z}))$
	\item the existence of a $U$ point implies the relation $R$ defines a dense linear ordering without endpoints on the $U$ points.
\end{enumerate}
\end{definition}

We now explain how changing $T$ to $T^+$ changes the Scott complexity counting function and the complexity of $T$.

\begin{lemma}
For any $T$ over a vocabulary $L$, $I(T^+,\Pinf{2})=I(T,\Pinf{2})+1$ and $I(T^+,\Gamma)=I(T,\Gamma)$ for $\Gamma\neq \Pinf{2}$.
\end{lemma}

\begin{proof}
Consider an $M\models T^+$. 
If $M\models \forall x\ \lnot U(x)$, then by the first axiom of $T^+$, $M$ is a model of $T$ with the $U$ and $R$ predicates empty.
If $M\models \exists x\ U(x)$, all relations from $L$ are empty, there is exactly one $\lnot U$ point $y$ that serves as the interpretation of all constants and the range of all functions in $L$, and the $R$-order of the $U$-points is isomorphic to the ordering of the rationals.
In particular, there is exactly one model with $M\models \exists x\ U(x)$ defined by the above axioms.
This model has Scott complexity $\Pinf{2}$ because it is infinite and the orbits are defined in a quantifier-free manner (by the unary relation $U$ along with the $R$-order).
In total, $T^+$ has a model for each model of $T$, plus one extra of Scott complexity $\Pinf{2}$ as required.
\end{proof}

\begin{lemma}
If $T$ is $\Pinf{2}$, then so is $T^+$.
\end{lemma}

\begin{proof}
This is immediate from counting quantifiers in the definition of $T^+$ and the fact that the dense linear ordering of the rationals has a $\Pinf{2}$ Scott sentence.
\end{proof}

We next provide a binary operation on theories that preserves their more complicated models, but preserves the property of having exactly one $\Pinf{2}$ model.

\begin{definition}
Fix two theories $T$ and $T'$ over vocabularies $L$ and $L'$ with exactly one $\Pinf{2}$ model each $\+M$ and $\+M'$.
Define the theory $T^\wedge T'$ over a two sorted version of the vocabulary $L\sqcup L'$ as follows:
\begin{enumerate}
	\item the first sort is a model of $T$
	\item the second sort is a model of $T'$
	\item either the first sort is isomorphic to $\+M$ or the second sort is isomorphic to $\+M'$.
\end{enumerate}
\end{definition}

\begin{lemma}
For any theories $T$ and $T'$ with unique $\Pinf{2}$ models $\+M$, respectively, $\+M'$, $I(T^\wedge T',\Pinf{2})=1$ and $I(T^\wedge T',\Gamma)=I(T,\Gamma)+I(T',\Gamma)$ for $\Gamma\neq \Pinf{2}$.
\end{lemma}

\begin{proof}
  Note that for a two-sorted structure with no interaction between the sorts, it is straightforward to see that $SC((\+N,\+K))=\max\{SC(\+N),SC(\+K)\}$.
By the third axiom for $T^\wedge T'$, either the first sort is a copy of $\+M$ or the second sort is a copy of $\+M'$.
We have $SC((\+M,\+K))=SC(\+K)$, so for each model of $T'$ there is a corresponding model with the same complexity in $T^\wedge T'$.
In the case that the first sort is not a copy of $\+M$, the second is a copy of $\+M'$ and $SC(\+N,\+M')=SC(\+N)$.
So, each model of $T'$ has a corresponding model with the same complexity in $T^\wedge T'$.
The only overlap between these two types of structures is $(\+M,\+M')$, so there is still exactly one $\Pinf{2}$ model, but at each other complexity $I(T^\wedge T',\Gamma)=I(T,\Gamma)+I(T',\Gamma)$ as desired.
\end{proof}

\begin{lemma}
If $T$ and $T'$ are $\Pinf{2},$ so is $T^\wedge T'$.
\end{lemma}

\begin{proof}
This is immediate from the definition of $T^\wedge T'$.
\end{proof}

We also observe the following general fact for theories with finitely many models.
This will restrict some of the possible behaviors of the Scott complexity counting function.

\begin{proposition}\label{prop:aboveDSinfGivesSinfPinf}
Let $T$ be $\Pinf{2}$ with finitely many models and $n\geq2$.
If there is $\+M$ such that $\+M\models T$ and $SC(\+M)>\dSinf{n}$ then there are models $\+M_1,\+M_2\models T$ with $SC(\+M_1)=\Sinf{n+1}$ and $SC(\+M_2)=\Pinf{n+1}$.
\end{proposition}

\begin{proof}
As in \cref{lem:smallnessandgreaterscimpliessc}, let $S_n=\{ \+S: \+S\models T \land SC(\+S)\leq \dSinf{n}\}$.
Note that $S_n/{\cong}$ is finite, so $S_n$ is a $\dSinf{n}$-definable.
Furthermore, as $Mod(T)\setminus S_n$ is  non-empty and both $\Pinf{n+1}$ and $\Sinf{n+1}$-definable, 
by \cref{lem:havingPimodels} and \cref{lem:havingSigmamodels}, there are $\+M_1,\+M_2\in Mod(T)\setminus S_n$ with $SC(\+M_1)\leq\Sinf{n+1}$ and $SC(\+M_2)\leq\Pinf{n+1}$.
As $\+M_1$ has no $\dSinf{n}$ Scott sentence, it must be that $SC(\+M_1)=\Sinf{n+1}$.
Symmetrically, $SC(\+M_2)=\Pinf{n+1}$, giving the desired result.
\end{proof}

We are now ready to classify the behavior of the Scott complexity counting function when the theory has exactly three models.

\begin{proposition}
A $\Pinf{2}$ formula $T$ with exactly three models has one of the following four Scott complexity spectra:
for some $n$ with $1\leq n\leq3$
\begin{align} \tag*{(1)-(3)}&I(T,\Pinf{2})=n,  I(T,\dSinf{2})=3-n, \text{and }  I(T,\Gamma)=0 \text{ if $\Gamma\not\in \{\Pinf{2}, \dSinf{2}\}$}\\
\tag{4}&I(T,\Pinf{2})=1,  I(T,\Pinf{3})=1,  I(T,\Sinf{3})=1,\text{and }  I(T,\Gamma)=0 \text{ if $\Gamma\not\in \{ \Pinf{2},\Pinf{3},\Sinf{3}\}$.}\end{align}
All possibilities are realized.
\end{proposition}

\begin{proof}
We first show that these are the only possibilities.
By \cref{lem:smallnessandgreaterscimpliessc} there must be at least one model of $T$ of complexity $\Pinf{2}$.
By \cref{prop:aboveDSinfGivesSinfPinf} if there is any model with complexity above $\dSinf{2},$ there must be at least one model of complexity $\Pinf{3}$ and $\Sinf{3}.$
As there are only three models in total, this forces the last spectrum to be the one in question.
Otherwise, all models are at most $\dSinf{2}$, and each of those possibilities is covered by the first three listed spectra.

We now demonstrate how to realize each of the Scott spectra.
Let $S$ be a $\Pinf{2}$ Scott sentence.
$S^{++}$ has $3$ models of Scott complexity $\Pinf{2}$ and no more.
Let $D$ be the theory from \cref{prop:2ModelsClassified} with exactly one structure of Scott complexity $\Pinf{2}$ and one structure of Scott complexity $\dSinf{2}$.
$D^+$ has $2$ models of Scott complexity $\Pinf{2}$ and one of complexity $\dSinf{2}$.
$D^\wedge D$ has one model of Scott complexity $\Pinf{2}$ and two of complexity $\dSinf{2}$.
Lastly, the theory $C$ from \cref{ex:cutModels} has three models; one of complexity $\Pinf{2}$, one of complexity $\Pinf{3}$, and one of complexity $\Sinf{3}$.
\end{proof}

\cref{prop:aboveDSinfGivesSinfPinf} also helps us obtain a general fact about theories with finitely many models, which partially generalizes what we have observed about theories with $2$ or $3$ models.

\begin{proposition}\label{prop:finiteHeight}
Say that $T$ is a $\Pinf{2}$ theory with $2N+1$ ($2N+2$) models.
If $\+M\models T$, then $SC(\+M)<\dSinf{N+2}$ ($SC(M)\leq\dSinf{N+2}$).
\end{proposition}

\begin{proof}
We begin with the case where $T$ has $2N+1$ models.
For the sake of contradiction, say that there is a $\+M\models T$ with $SC(\+M)\geq\dSinf{N+2}>\dSinf{N+1}$.
By \cref{prop:aboveDSinfGivesSinfPinf}, this means that for $2\leq k\leq N+1,$ there are models $\+M^k_1$ and $\+M^k_2$ with $SC(\+M^k_1)=\Sinf{k+1}$ and $SC(\+M^k_2)=\Pinf{k+1}$.
Furthermore, by \cref{lem:smallnessandgreaterscimpliessc} there must be at least one model $\+A\models T$ of complexity $\Pinf{2}$.
In total, the $\+M^k_i$ and $A$ are $2N+1$ models, all of complexity strictly less than $\dSinf{N+2}$.
This means that $\+M$ is not among these models, so $T$ has at least $2N+2$ models, a contradiction.

We now look at the case where $T$ has $2N+2$ models.
This will proceed similarly to the above case.
For the sake of contradiction, say that there is $\+M\models T$ with $SC(\+M)>\dSinf{N+2}$.
By \cref{prop:aboveDSinfGivesSinfPinf}, this means that for $2\leq k\leq N+2,$ there are models $\+M^k_1$ and $\+M^k_2$ with $SC(\+M^k_1)=\Sinf{k+1}$ and $SC(\+M^k_2)=\Pinf{k+1}$.
Furthermore, by \cref{lem:smallnessandgreaterscimpliessc} there must be at least one model $\+A\models T$ of complexity $\Pinf{2}$.
In total, the $\+M^k_i$ and $\+A$ are $2N+3$ models, all of different complexities, a contradiction to the assumption that there are only $2N+2$ models.
\end{proof}

Repeating the same argument at level $\alpha$ immediately yields the following corollary.

\begin{corollary}
Say that $T$ is a $\Pinf{\alpha}$ theory with $2N+1$ ($2N+2$) models.
If $\+M\models T$, then $SC(\+M)<\dSinf{\alpha+N}$ ($SC(\+M)\leq\dSinf{\alpha+N}$).
\end{corollary}
\ehrenfeuchtomegavaught

The $\omega$-Vaught's conjecture is an infinitary analog of Martin's conjecture, which is still open for Ehrenfeucht theories \cite[Problem 7]{PT25}. As every first-order theory is trivially a $\Pinf{\omega}$ theory, \cref{cor:ehrenfeuchtomegavaught} gives that $vo(T)<\omega\cdot2$ for first-order, Ehrenfeucht $T$. It improves a result of Wagner~\cite{Wag82} who showed that first-order Ehrenfeucht theories $T$ have $vo(T)<\omega^2$.

\begin{corollary}
Fix $N\in\omega$. There are only finitely many possible Scott complexity spectra for $\Pinf{2}$ theories with $N$ models.
\end{corollary}

\begin{proof}
By \cref{prop:finiteHeight}, there are only finitely many Scott complexities realized by models of such theories $T$.
Furthermore, for each such complexity $\Gamma$, certainly $0\leq I(T,\Gamma)\leq N$, so there are only finitely many possible arrangements.
\end{proof}

This motivates the definition of a spectrum counting function as follows.

\begin{definition}
  The \define{spectrum counting function} $F:\omega\to\omega$ is so that $F(N)$ is the number of possible Scott complexity spectra for $\Pinf{2}$ theories with $N$ models.
\end{definition}

We have calculated above that $F(2)=2$ and $F(3)=4$.
In general, we can comfortably say that $F(2N+1)<(2N+1)^{3N+1}$ and $F(2N+2)<(2N+2)^{3N+2}$ by the very logic in the above corollary.
We expect the actual values of the function to be far smaller, however.
For example, we note that \cref{prop:aboveDSinfGivesSinfPinf} does not give the only restriction on the possible spectra.
Indeed, we will see that the roles of the $\Pinf{k+1}$ and $\Sinf{k+1}$ models are not entirely symmetric, even in the finitely many models scenario.

\begin{proposition}\label{prop:finiteSpectralInequality}
If $T$ is a $\Pinf{2}$ Ehrenfeucht theory, then for every $k\geq1$
\[I(T,\Pinf{k+1}) \leq I(T,\Sinf{k+1}).\]
\end{proposition}

\begin{proof}
Divide the models of $T$ into finitely many $k$-back-and-forth classes.
As in \cref{lem:smallEasybnfDescription}, let $E(\+M,k):=\{\+N : \+N\models T \land \+N\equiv_k \+M\}$ for a fixed $\+M$.
Using the notation of \cref{lem:smallEasybnfDescription}, we have that 
\[\+N\in E(\+M,k) \iff \+N \models T \land \+N\models\bigwwedge_{i>0} \lnot\psi_i,\]
where the conjunction ranges over all $k$-back-and-forth classes and $\psi_i$ is either $\Pinf{k}$ or $\Sinf{k}$.
In particular, as we only have finitely many classes, the set $E(\+M,k)$ is both $\Pinf{k+1}$ and $\Sinf{k+1}$-definable.
This means that if we have a model $\+N$ with $SC(\+N)=\Pinf{k+1}$, we must not have $\{\+N\}=E(\+N,k)$.
Therefore, there must be another model $\+A\in E(\+N,k)$ and thus the theory $\+A\in E(\+N,k)\land \+A\not\cong \+N$ is consistent. Furthermore, it is $\Sinf{k+1}$.
Hence, there must be a model $\+A$ with $\+A\equiv_k \+N$ and $SC(\+A)\leq\Sinf{k+1}$.
If $SC(\+A)\leq\Pinf{k+1}$, then $\+A\in E(\+N,k)\implies \+N\cong \+A$ by \cref{cor:onlyOneSRalphaPerClass}, a contradiction.
Therefore, $SC(\+A)\geq\Sinf{k+1}$ and so $SC(\+A)=\Sinf{k+1}$.
In other words, for each model with $SC(\+N)=\Pinf{k+1}$, there is at least one other model in its $k$-back-and-forth class, and this model is of complexity $\Sinf{k+1}$, proving the claimed result.
\end{proof}

For the sake of completeness, we point out the following corollary to the above methods.
The corollary puts a finer point on the necessary abundance of models with $\Sigma$ complexity.

\begin{corollary}
If $T$ is a $\Pinf{2}$ Ehrenfeucht theory and $\+M\models T$ with $SC(\+M)=\Pinf{k+1}$, then there exists some $\+A$ with $\+A\equiv_k \+N$ and $SC(\+A)=\Sinf{k+1}$.
\end{corollary}

This corollary is in stark contrast to \cref{cor:onlyOneSRalphaPerClass}.
In this setting, $\Pinf{k+1}$ complexity structures are always $k$-equivalent to $\Sinf{k+1}$ complexity structures, but never $k$-equivalent to other $\Pinf{k+1}$ complexity structures.
As indicated above, the property in \cref{prop:finiteSpectralInequality} is not symmetric.
We demonstrate this with the following standard set of examples.

\begin{example}\label{ex:kcutModels}
Let $C_k$ be a theory over the vocabulary containing one binary relation symbol $<$, infinitely many constants $c_i$, and $k$ unary predicates $\{U_i\}_{i<k}$.
$C_k$ states that
\begin{enumerate}
	\item $<$ is a dense linear ordering without endpoints,
	\item the predicates $U_i$ partition the domain,
	\item $U_0(c_i)$ holds for every $i$,
	\item $c_i<c_j$ exactly when $i<j$.
\end{enumerate}
\end{example}

It is well known that $C_k$ has exactly $k+2$ models (see for example~\cite[Exercise 2.5.28]{marker2002}):
\begin{enumerate}
	\item $\+U$ in which the $c_i$ are unbounded
	\item $\+L_j$ in which the $c_i$ are bounded and have a limit realized by a $U_j$ point in $L$
	\item $\+N$ in which the $c_i$ are bounded but have no limit realized in $N$.
\end{enumerate}
\begin{proposition}\label{prop:kcutAnalysis}
$I(C_k,\Pinf{2})=1$, $I(C_k,\Pinf{3})=1$, and $I(C_k,\Sinf{3})=k$.
\end{proposition}

\begin{proof}
We note that $SC(\+U)=\Pinf{2}$ and $SC(\+N)=\Pinf{3}$.
This follows by a slight modification of the arguments already seen in \cref{ex:cutModels}.
By \cref{prop:aboveDSinfGivesSinfPinf}, one of the $\+L_j$ must have $SC(\+L_j)=\Sinf{3}$.
The apparent symmetry allows us to conclude that for all $j$, $SC(\+L_j)=\Sinf{3}$, showing the desired claim.
\end{proof}

When it comes to Scott spectra with support strictly below $\dSinf{3}$, these observations and examples explain and exhaust all possible behavior.

\begin{theorem}\label{thm:smallSupportChar}
Let $Supp({<}\dSinf{3})$ be the set of Scott spectra of $\Pinf{2}$ theories with finitely many models, all with Scott complexity strictly below $\dSinf{3}.$
$Supp({<}\dSinf{3})$ contains exactly the spectra subject to the following constraints:
\begin{enumerate}
	\item $I(T,\Pinf{3}) \leq I(T,\Sinf{3})$,
	\item $I(T,\Pinf{3})>0\implies I(T,\Sinf{3})>0$,
	\item $I(T,\Pinf{2})>0$.
\end{enumerate}
\end{theorem}
\begin{proof}
We first note that the three listed constraints must hold for spectra in $Supp(<\dSinf{3})$.
The first follows from \cref{prop:finiteSpectralInequality}.
The second follows from \cref{prop:aboveDSinfGivesSinfPinf}.
The last follows from \cref{lem:smallnessandgreaterscimpliessc}.

We now show that we can obtain every spectrum subject to our constraints.
We construct a $\Pinf{2}$ theory $T$ with finitely many models and $I(T,\Pinf{2})=a$, $I(T,\dSinf{2})=b$, $I(T,\Pinf{3})=c$, and $I(T,\Sinf{3})=d$ where $a>0$, $c\leq d$ and $d>0\implies c>0$.
We use $T^{k+}$ to denote the plus operator applied to $T$, $k$ times and use $\bigwedge_rT$ to denote $T^{\wedge}\cdots^\wedge T$, $r$ times.
By convention, $T^{0+}=T$ and $\bigwedge_0T$ is a $\Pinf{2}$ Scott sentence.
Furthermore, we take the $C_k$ from \cref{ex:kcutModels} and $D$ the non-trivial example from \cref{prop:2ModelsClassified}. 
If $d>0$, consider the theory
\[ T= (C_{d-c+1}{}^\wedge \bigwedge_{c-1}C_1{}^\wedge\bigwedge_b D)^{(a-1)+}.\]
Note that $c,a>0$, so $T$ is well defined.
By construction,
\begin{itemize}
	\item $I(T,\Pinf{2})=1+(a-1)=a$,
	\item $I(T,\dSinf{2})= b$,
	\item $I(T,\Pinf{3})=1+(c-1)=c$,
	\item $I(T,\Sinf{3})=(d-c+1)+(c-1)=d$,
\end{itemize}
as desired.
If $d=0, c=0$ as well.
In this case,
\[ T= (\bigwedge_b D)^{(a-1)+}, \]
and $I(T,\Pinf{2})=1+(a-1)=a$ while $I(T,\dSinf{2})= b$, as desired.
\end{proof}
These constructions give us a lot of insight into theories with four models.
\begin{proposition}
$7\leq F(4)\leq 8$.
\end{proposition}

\begin{proof}
By \cref{thm:smallSupportChar}, the following spectra are all possible.
For $1\leq n\leq 4$ and $1\leq m\leq 2$,
\begin{align*} &I(T,\Pinf{2})=n , I(T,\dSinf{2})=4-n , \text{and}  I(T,\Gamma)=0 \text{ otherwise;}\\
 &I(T,\Pinf{2})=1 , I(T,\Pinf{3})=1 , I(T,\Sinf{3})=2 , \text{and} I(T,\Gamma)=0 \text{ otherwise;}\\
 &I(T,\Pinf{2})=m ,  I(T,\dSinf{2})=2-m ,  I(T,\Pinf{3})=1 , I(T,\Sinf{3})=1 , \text{and} I(T,\Gamma)=0 \text{ otherwise}\end{align*}
This gives that $7\leq F(4)$.
By \cref{thm:smallSupportChar}, these are the only spectra supported strictly below $\dSinf{3}$.
If there is a model of $T$ with Scott complexity at least $\dSinf{3}$, by \cref{prop:finiteHeight} it has Scott complexity exactly $\dSinf{3}$.
Furthermore, by \cref{prop:aboveDSinfGivesSinfPinf}, $T$ has a model of Scott complexity $\Sinf{3}$ and $\Pinf{3}$.
By \cref{lem:smallnessandgreaterscimpliessc}, $T$ also has a model of Scott complexity $\Pinf{2}$.
This means that it must be that 
\[ I(T,\Pinf{2})=1 ,  I(T,\dSinf{3})=1 ,  I(T,\Pinf{3})=1 , I(T,\Sinf{3})=1 ,\text{ and } I(T,\Gamma)=0 \text{ otherwise.}\]
Therefore, there is only one Scott spectrum that is not accounted for, and $F(4)\leq 8$.
\end{proof}

We leave open the exact value of $F(4)$ and, in general, for $F(n)$ with $n\geq 4$.

\begin{question}
What is the value of $F(4)$ and of $F(n)$ for $n\geq 4$? Can a $\Pinf{2}$ Ehrenfeucht theory have a model of any Scott complexity corresponding to a finite Scott rank?
\end{question}

We suspect that $F(4)=8$.
The reason for this is that there are examples of $\Pinf{2}$ Ehrenfeucht theories with a model as complex as $\dSinf{3}$, which is the maximum complexity needed to realize the unaccounted for spectrum.
This means that there is no reason that models of Ehrenfeucht theories must avoid the needed complexities, or, in particular, must be simpler than the needed models.
We give an example of a $\Pinf{2}$ Ehrenfeucht theory with a $\dSinf{3}$ model below.

\begin{proposition}
There is a $\Pinf{2}$ theory $T$ with exactly 6 models where one of the models has Scott complexity $\dSinf{3}$.
\end{proposition}

\begin{proof}
$T$ will resemble a duplication of the theory from \cref{ex:cutModels}.
$T$ will be over the vocabulary $<,E,\{P_i\}_{i\in\omega}$ where $<$ and $E$ are binary relations and each of the $P_i$ are unary relations.
$T$ states that:
\begin{enumerate}
	\item $E$ is an equivalence relation with exactly two classes.
	\item $<$ linearly orders each equivalence class and does not hold for any pair of elements not in the same equivalence class.
	\item The order type of each equivalence class is a dense linear ordering without endpoints.
	\item Each $P_i$ holds for exactly two points that are in different equivalence classes.
	\item If $P_i(x)\land P_{i+1}(y) \land E(x,y)$, then $x<y$.
\end{enumerate}
It is straightforward to confirm that $T$ is $\Pinf{2}$ given the above description.
Focusing on a single equivalence class, the structure is a dense linear ordering with exactly one $P_i$ point for each $i$ that increases as $i$ increases.
In other words, each equivalence class contains a model of the Ehrenfeucht theory $C$ with three models (\cref{ex:cutModels}).
Thus, there are three choices for the isomorphism type of each equivalence class, and so $T$ has six models.
Using the notation from \cref{ex:cutModels}, the models are of the form $(\+U,\+U);(\+U,\+N);(\+U,\+L);(\+N,\+N);(\+N,\+L)$ and $(\+L,\+L)$.

The model of Scott complexity $\dSinf{3}$ is $(\+N,\+L)$.
To see this first note that $(\+N,\+L)$ has a $\dSinf{3}$ Scott sentence:
The $\Sinf{3}$ formula 
\[\exists x\psi(x):=\exists x (\forall y<x) (\exists z>y) \bigvvee_i P_i(z),\]
is true of $\+L$ but not of $\+N$.
We also have the $\Sinf{2}$ formula
\[\exists x\varphi(x):=\exists x(\forall y>x) \bigwwedge_i \lnot P_i(y),\]
which is true of both $\+L$ and $\+N$ yet not true of $\+U$.
Modifying these statements slightly gives the formula
\[\exists x\psi(x) \land \lnot\exists y,z \ \lnot E(y,z) \land \psi(y) \land \psi(z) \land \exists w,v \ \lnot E(w,v) \land \varphi(w) \land \varphi(v).\]
This $\dSinf{3}$ formula states that neither equivalence class is isomorphic to $\+U$ and exactly one is isomorphic to $\+L$.
Together with $T$, the formula gives the desired Scott sentence for $(\+N,\+L)$.

We now show that $(\+N,\+L)$ does not have a simpler Scott sentence.
This will be accomplished by first showing that $\+N\geq_3 \+L$.
Say that this was not the case.
Then there would be a $\Pinf{3}$ formula $\varphi$ such that $\+L\models\varphi$ yet $\+N\models \lnot \varphi$.
We would then have the $\Pinf{3}$ formula $\varphi\land\lnot U\land C$ that would be true of exactly $\+L$ and no other models.
However, as we saw in \cref{ex:cutModels}, $\+L$ has Scott complexity $\Sinf{3}$, a contradiction.
Therefore, $\+N\geq_3 \+L$.
Playing the back-and-forth game independently on each equivalence class, we may now observe that
\[(\+L,\+L)\leq_3(\+N,\+L)\leq_3(\+N,\+N).\]
This means that $(\+N,\+L)$ does not have a $\Pinf{3}$ or $\Sinf{3}$ Scott sentence, and so its simplest Scott sentence is $\dSinf{3}$, as desired.
\end{proof}

\section{Prime models in complete first-order theories}
Vaught's conjecture was originally stated in the context of complete first-order theories.
We now examine some consequences of our tools in this context.
In particular, we examine the general Scott analysis for prime models of first-order theories.
Perhaps unsurprisingly, prime models often have fairly tame Scott complexity.
Below, we say that a first-order theory is $\Pi_m$ if all of its axioms are $\Pi_m$.
We also use the following standard definition.

\begin{definition}
    A complete theory $T$ is \define{$\omega$-stable} if for every $\+M\models T$ and $A\subseteq M$, $\+M$ realizes at most $\max(\aleph_0,|A|)$ many types with parameters from $A$. 
\end{definition}

Note that $\omega$-stable theories only have countably many first-order types (over the empty set) realized among their models.
Furthermore, it is well known that $\omega$-stable theories have prime models \cite[Chapter 7]{marker2002}.

\begin{proposition}\label{prop:Pi2PrimeModel}
Let $T$ be a complete $\Pi_2$, first-order theory that is $\omega$-stable or has fewer than continuum many countable models.
The prime model $\+P$ must have Scott complexity $\Pinf{2}$.
\end{proposition}

\begin{proof}
If $T$ is $\omega$-stable or has fewer than continuum many countable models, then it has only countably many first-order types realized among its countable models.
Note that all $\Sinf1$ formulas true of a tuple are disjunctions of finitary $\Sigma_1$ formulas.
This means that the finitary type of a tuple determines its $\Sigma_1$ type, and so $T$ is $\Sigma_1$-small.
By \cref{lem:havingPimodels}, $T$ must have a model $\+M$ of Scott complexity $\Pinf{2}$.
We have that $\+P\preceq \+M$, say via the elementary embedding $\sigma$.
In particular, $\+P$ embeds into $\+M$ in a way that preserves all $\Sinf{1}$ formulas.
We claim that $\+P\geq_2\+M$.
The $\exists$-player may play along the elementary embedding $\sigma$ on the first move of the back-and-forth game with the guarantee that $(\+P,\ba)\leq_1(\+M,\sigma(\ba))$ for their winning strategy.
Because $\+M$ has Scott complexity $\Pinf{2}$, we see that $\+P\cong \+M$, and so $\+P$ has Scott complexity $\Pinf{2}$, as desired.
\end{proof}
%
\begin{proposition}\label{prop:pi2impprime}
  If $\+A$ is of Scott complexity $\Pinf{2}$, then $\+A$ is a prime model of its first-order theory.
\end{proposition}
\begin{proof}
  The structure $\+A$ has existentially definable orbits and thus every $\Pinf{1}$-type is existentially supported. In general, for any $\bar a$
  \[ aut(\bar a)=\{ \bar c: \Pinf{1}\-tp(\bar a)=\Pinf{1}\-tp(\bar c)\}=\{ \bar c: \Pi_{1}\-tp(\bar a)=\Pi_{1}\-tp(\bar c)\}\supseteq \{ \bar c: tp(\bar a)=tp(\bar c)\},\] 
  and by the minimality of orbits it follows that any element sharing $\bar a$'s type must be automorphic. This shows that every first-order type is isolated and hence $\+A$ is countable atomic, and therefore also prime.
\end{proof}
We note that \cref{prop:Pi2PrimeModel,prop:pi2impprime} has been independently proven by Jason Block.
\begin{corollary}
A complete, first-order theory has at most one model of Scott complexity $\Pinf{2}$.
\end{corollary}

These observations limit the coding power of complete first-order theories.
Below, an effective bi-interpretation is a strong manner of transforming models of one theory into the models of another theory while preserving all salient properties studied in computable structure theory.
This includes invariants like Scott rank and Scott complexity.
More information about effective bi-interpretation can be found in \cite[Chapter VI.4]{Part1}.
Our corollary only needs the fact that effective bi-interpretation preserves (unparamaterized) Scott rank, so it also applies to any notion of reduction with this property.

\begin{corollary}
If $T$ is a complete first-order theory, then $T$ is not on top for effective bi-interpretation or any other reduction notion that preserves (unparamaterized) Scott rank.
\end{corollary}

\begin{proof}
As seen in \cref{prop:Pi2PrimeModel}, $T$ has only one $1$-back-and-forth equivalence class.
Therefore, $T$ can have at most one model of Scott rank 1 by \cref{prop:uniqueminSR}.
Therefore, $T$ cannot accept an effective bi-interpretation or any other reduction that preserves Scott rank from a theory with more than one model of Scott rank 1.
\end{proof}

Note that the $\omega$-stability assumption is needed in \cref{prop:Pi2PrimeModel}; not all $\Pi_2$ first-order theories have a prime model.
The following standard example demonstrates this claim.

\begin{proposition}
There is a complete, $\Pi_2$ axiomatizable first-order theory with no models of Scott complexity $\Pinf{2}$.
\end{proposition}

\begin{proof}
  Let $T$ be a theory over infinitely many unary predicates $P_i$,  and given $I,J\subset_{fin}\omega$ let
\[ \psi_{I,J} := \exists x \bigwedge_{i\in I} P_i(x) \land \bigwedge_{j\in J} P_j(x).\]
An element $x\in T$ can be thought of as a path through $2^{<\omega}$ declaring which predicates hold or do not hold of $x$.
By its presented form, $T$ is $\Pi_2$; indeed, all of its axioms are even $\Sigma_1$.
Note that $T$ is complete as a straightforward argument demonstrates that it admits quantifier elimination.

Consider $\+M\models T$.
Say that it has a $\Pinf{2}$ Scott sentence.
This means that every element has an orbit isolated by a (first-order) $\Sigma_1$ formula.
By quantifier elimination, each orbit is isolated by a quantifier-free formula.
Consider $x\in \+M$ and its isolating formula $\varphi$.
$\varphi$ can only declare that finitely many predicates in the set $I'$ do hold of $x$ and finitely many predicates in the set $J'$ do not hold of $x$.
Let $n\notin I'\cup J'$.
If $\+M\models P_n(x)$ then $\psi_{I',J'\cup\{n\}}$ guarantees the existence of an element with $\varphi$ that is not automorphic of $x$.
Similarly, If $\+M\models \lnot P_n(x)$ then $\psi_{I'\cup\{n\},J'}$ guarantees the existence of an element with $\varphi$ that is not automorphic of $x$.
This yields the desired contradiction, so $\+M$ has no models of Scott complexity $\Pinf{2}$.
\end{proof}

\cref{prop:Pi2PrimeModel} iterates up the arithmetic hierarchy in the expected manner.

\pinprime

\begin{proof}
Say that $T$ is defined over the vocabulary $L$.
We Morleyize $L$ to include new predicates for all $\Sigma_n$ sets.
This can be done in a manner where the definitions of each new predicate are $\Pi_2$ (see \cite{Part2} Definition II.25).
In particular, we obtain a $\Pi_2$ theory $T'$ that is an expansion by definitions of $T$ over the vocabulary $L'\supseteq L$.
If $T$ is $\omega$-stable, as $T'$ is an expansion by definitions, $T'$ is also $\omega$-stable.
If $T$ is complete and has fewer than continuum many countable models, $T'$ similarly is complete and has fewer than continuum many countable models.
By \cref{prop:Pi2PrimeModel}, $T'$ has a prime model $P'$ with complexity $\Pinf{2}$.
Let $P$ be the reduct of $P'$ with $P\models T$.
Consider the type of some $x\in P$.
In $P'$, the $L'$-type of $x$ is isolated by some $\varphi\in L'$, because $P'$ is a prime model.
Over $T'$, we can expand the $L'$ predicates in $\varphi$ to their definitions in $L$.
This gives a formula $\hat{\varphi}\in L$ that isolates the $L'$-type of $x$ in $P'$.
In particular, $\hat{\varphi}$ is an $L$-formula that isolates the $L$-type of $x$ in $P$.
Therefore, every type realized in $P$ is isolated, so $P$ is prime.
Consider now $\chi$, the $\Pinf{2}$ Scott sentence for $P$.
Proceeding as above, we may define $\hat{\chi}$, an $L$-formula equivalent to $\chi$, and therefore also a Scott sentence for $P$.
Furthermore, as each quantifier-free formula in $\chi$ is replaced by a formula of quantifier depth $n$, $\hat{\chi}$ is  $\Pinf{n+2}$.
Lastly, $\chi$ is an $L$-Scott sentence for $P'$, so it is also a Scott sentence for $P$, as desired.
\end{proof}

In general, the Scott rank of a prime model may be as high as $\omega$, as the automorphism orbits isolated by the first-order formulas isolating the first-order type of the tuple may be $\Sigma_n$ for any $n$. An example of models where this Scott rank is attained are prime models of non-standard completions of Peano arithmetic~\cite{montalban2024}, see also~\cite{GLRS}.

%
\printbibliography

@article{miller1983,
  title = {On the {{Borel}} Classification of the Isomorphism Class of a Countable Model.},
  author = {Miller, Arnold W.},
  date = {1983},
  journaltitle = {Notre Dame Journal of Formal Logic},
  volume = {24},
  number = {1},
  pages = {22--34}
}

@article{Sil80,
	author = {Silver, Jack H.},
	coden = {AMLOAD},
	fjournal = {Annals of Mathematical Logic},
	issn = {0003-4843},
	journal = {Ann. Math. Logic},
	mrclass = {03E15},
	mrnumber = {MR568914 (81d:03051)},
	mrreviewer = {Keith Devlin},
	number = {1},
	pages = {1--28},
	title = {Counting the number of equivalence classes of {B}orel and coanalytic equivalence relations},
	volume = {18},
	year = {1980}}

@article{richter1981,
	author = {Richter, Linda J.},
	date = {1981},
	doi = {10.2307/2273222},
	journaltitle = {The Journal of Symbolic Logic},
	number = {04},
	pages = {723--731},
	title = {Degrees of Structures},
	url = {http://journals.cambridge.org/abstract_S0022481200045035},
	urldate = {2015-06-06},
	volume = {46}}

@book{marker2002,
	author = {Marker, David},
	date = {2002},
	doi = {10.1007/b98860},
	isbn = {978-0-387-98760-6},
	publisher = {Springer Science \& Business Media},
	shorttitle = {Model Theory},
	title = {Model Theory: An Introduction},
	url = {https://books.google.at/books?hl=en&lr=&id=QieAHk--GCcC&oi=fnd&pg=PA1&dq=marker+model+theory&ots=Z2PxCJt1kE&sig=DcehxLnj2h9T6YRcbrxIviS68R8},
	urldate = {2015-06-03}}

@article{montalban2024,
	author = {Montalb\'an, Antonio and Rossegger, Dino},
	date = {2024},
	doi = {10.1017/jsl.2023.43},
	eprint = {2208.01697},
	eprinttype = {arXiv},
	issn = {0022-4812, 1943-5886},
	journaltitle = {The Journal of Symbolic Logic},
	number = {4},
	pages = {1703--1719},
	title = {The Structural Complexity of Models of Arithmetic},
	volume = {89}}

@incollection{Ste78,
	address = {Berlin},
	author = {Steel, John R.},
	booktitle = {Cabal {S}eminar 76--77 ({P}roc. {C}altech-{UCLA} {L}ogic {S}em., 1976--77)},
	mrclass = {03C15 (03C70 03E15)},
	mrnumber = {526920 (81b:03036)},
	mrreviewer = {John Rosenthal},
	pages = {193--208},
	publisher = {Springer},
	series = {Lecture Notes in Math.},
	title = {On {V}aught's conjecture},
	volume = {689},
	year = {1978}}

@article{Khi04,
	author = {Khisamiev, A. N.},
	fjournal = {Rossi\u\i skaya Akademiya Nauk. Sibirskoe Otdelenie. Institut Matematiki im. S. L. Soboleva. Sibirski\u\i\ Matematicheski\u\i\ Zhurnal},
	issn = {0037-4474},
	journal = {Sibirsk. Mat. Zh.},
	mrclass = {03D45},
	mrnumber = {MR2048764 (2005f:03061)},
	mrreviewer = {Andrzej Orlicki},
	number = {1},
	pages = {211--228},
	title = {On the {E}rshov upper semilattice {$L_E$}},
	volume = {45},
	year = {2004}}

@article{ACM24,
	author = {Alvir, Rachael and Csima, Barbara F. and MacLean, Luke},
	journal = {Preprint},
	title = {Scott Complexity of Reduced Abelian p-Groups},
	volume = {https://arxiv.org/abs/2407.06940},
	year = {2024}}

@article{HM,
	author = {Harris, Kenneth and Montalb{\'a}n, Antonio},
	coden = {TAMTAM},
	doi = {10.1090/S0002-9947-2011-05331-6},
	fjournal = {Transactions of the American Mathematical Society},
	issn = {0002-9947},
	journal = {Trans. Amer. Math. Soc.},
	mrclass = {03D80 (03C57)},
	mrnumber = {2846355},
	number = {2},
	pages = {827--866},
	title = {On the {$n$}-back-and-forth types of {B}oolean algebras},
	url = {http://dx.doi.org/10.1090/S0002-9947-2011-05331-6},
	volume = {364},
	year = {2012}}

@article{Sac07,
	author = {Sacks, Gerald E.},
	coden = {NDJFAM},
	doi = {10.1305/ndjfl/1172787542},
	fjournal = {Notre Dame Journal of Formal Logic},
	issn = {0029-4527},
	journal = {Notre Dame J. Formal Logic},
	mrclass = {03C70 (03D60)},
	mrnumber = {2289894 (2008a:03062)},
	mrreviewer = {Barbara Majcher-Iwanow},
	number = {1},
	pages = {5--31},
	title = {Bounds on weak scattering},
	url = {http://dx.doi.org/10.1305/ndjfl/1172787542},
	volume = {48},
	year = {2007}}

@article{HT22,
	author = {Harrison-Trainor, Matthew},
	doi = {10.1017/bsl.2021.62},
	fjournal = {The Bulletin of Symbolic Logic},
	issn = {1079-8986,1943-5894},
	journal = {Bull. Symb. Log.},
	mrclass = {03D45 (03C57 03C70)},
	mrnumber = {4402053},
	mrreviewer = {Alexandra\ Andreeva\ Soskova},
	number = {1},
	pages = {71--103},
	title = {An introduction to the {S}cott complexity of countable structures and a survey of recent results},
	url = {https://doi.org/10.1017/bsl.2021.62},
	volume = {28},
	year = {2022}}

@article{Mor70,
	author = {Morley, Michael},
	fjournal = {The Journal of Symbolic Logic},
	issn = {0022-4812},
	journal = {J. Symbolic Logic},
	mrclass = {02.50},
	mrnumber = {0288015 (44 \#5213)},
	mrreviewer = {F. R. Drake},
	pages = {14--18},
	title = {The number of countable models},
	volume = {35},
	year = {1970}}

@book{Part1,
	author = {Montalb\'{a}n, Antonio},
	doi = {10.1017/9781108525749},
	isbn = {978-1-108-42329-8},
	label = {Part 1},
	mrclass = {03-02 (03C57)},
	mrnumber = {4274028},
	pages = {xxii+190},
	publisher = {Cambridge University Press, Cambridge; Association for Symbolic Logic, Ithaca, NY},
	series = {Perspectives in Logic},
	title = {Computable structure theory---within the arithmetic},
	url = {https://doi.org/10.1017/9781108525749},
	year = {2021}}

@article{Wag82,
	author = {Wagner, C. M.},
	doi = {10.1016/0003-4843(82)90015-8},
	fjournal = {Annals of Mathematical Logic},
	issn = {0003-4843},
	journal = {Ann. Math. Logic},
	mrclass = {03C52 (03C15 03C45)},
	mrnumber = {661477},
	number = {1},
	pages = {47--67},
	title = {On {M}artin's conjecture},
	url = {https://doi.org/10.1016/0003-4843(82)90015-8},
	volume = {22},
	year = {1982}}

@incollection{Vau61,
	address = {Oxford},
	author = {Vaught, R. L.},
	booktitle = {Infinitistic {M}ethods ({P}roc. {S}ympos. {F}oundations of {M}ath., {W}arsaw, 1959)},
	mrclass = {02.50},
	mrnumber = {0186552 (32 \#4011)},
	mrreviewer = {G. Fuhrken},
	pages = {303--321},
	publisher = {Pergamon},
	title = {Denumerable models of complete theories},
	year = {1961}}

@article{HM77,
	author = {Harnik, V. and Makkai, M.},
	doi = {10.2307/2041293},
	fjournal = {Proceedings of the American Mathematical Society},
	issn = {0002-9939},
	journal = {Proc. Amer. Math. Soc.},
	mrclass = {02H10 (02B25)},
	mrnumber = {472506},
	mrreviewer = {John P. Burgess},
	number = {2},
	pages = {309--314},
	title = {A tree argument in infinitary model theory},
	url = {https://doi.org/10.2307/2041293},
	volume = {67},
	year = {1977}}

@incollection{Sac83,
	address = {Amsterdam},
	author = {Sacks, Gerald E.},
	booktitle = {Southeast {A}sian conference on logic ({S}ingapore, 1981)},
	mrclass = {03C15},
	mrnumber = {723338 (85i:03095)},
	mrreviewer = {David E. Marker},
	pages = {185--195},
	publisher = {North-Holland},
	series = {Stud. Logic Found. Math.},
	title = {On the number of countable models},
	volume = {111},
	year = {1983}}

@article{CGHT,
	author = {Ruiyan Chen and David Gonzalez and Matthew Harrison-Trainor},
	journal = {Transactions of the American Mathematical Society},
	title = {Optimal syntactic definitions of back-and-forth types},
	volume = {To appear},
	year = {2026+}}

@incollection{Karp,
	author = {Karp, Carol R.},
	booktitle = {Theory of {M}odels ({P}roc. 1963 {I}nternat. {S}ympos. {B}erkeley)},
	mrclass = {02.35},
	mrnumber = {209132},
	mrreviewer = {J.\ E.\ Fenstad},
	pages = {407--412},
	publisher = {North-Holland, Amsterdam},
	title = {Finite-quantifier equivalence},
	year = {1965}}

@article{MonICM,
	author = {Antonio Montalb{\'a}n},
	journal = {Proceedings of ICM 2014},
	pages = {79-101},
	title = {Computability theoretic classifications for classes of structures},
	url = {ClassesOfStructures.pdf},
	volume = {2},
	year = {2014}}

@article{AGNHTT,
	author = {Alvir, Rachael and Greenberg, Noam and Harrison-Trainor, Matthew and Turetsky, Dan},
	doi = {10.1017/jsl.2021.4},
	fjournal = {The Journal of Symbolic Logic},
	issn = {0022-4812,1943-5886},
	journal = {J. Symb. Log.},
	mrclass = {03D45 (03C57 03C70)},
	mrnumber = {4362932},
	mrreviewer = {Wesley\ Calvert},
	number = {4},
	pages = {1706--1720},
	title = {Scott complexity of countable structures},
	url = {https://doi.org/10.1017/jsl.2021.4},
	volume = {86},
	year = {2021}}

@article{MonSR,
	author = {Montalb{\'a}n, A.},
	doi = {10.1090/proc/12669},
	fjournal = {Proceedings of the American Mathematical Society},
	issn = {0002-9939},
	journal = {Proc. Amer. Math. Soc.},
	mrclass = {03D45 (03C57 03C75 03D60)},
	mrnumber = {3411157},
	mrreviewer = {Rodney G. Downey},
	number = {12},
	pages = {5427--5436},
	title = {A robuster Scott rank},
	url = {scottRank.pdf},
	volume = {143},
	year = {2015}}

@incollection{Sco65,
	address = {Amsterdam},
	author = {Scott, Dana},
	booktitle = {Theory of {M}odels ({P}roc. 1963 {I}nternat. {S}ympos. {B}erkeley)},
	mrclass = {02.35},
	mrnumber = {0200133 (34 \#32)},
	mrreviewer = {E. Engeler},
	pages = {329--341},
	publisher = {North-Holland},
	title = {Logic with denumerably long formulas and finite strings of quantifiers},
	year = {1965}}

@misc{PT25,
	archiveprefix = {arXiv},
	author = {Anand Pillay and Predrag Tanovi{\'c}},
	eprint = {2508.06854},
	primaryclass = {math.LO},
	title = {The number of countable models of first-order theories},
	url = {https://arxiv.org/abs/2508.06854},
	year = {2025}}

@article{GM23,
	author = {Gonzalez, David and Montalb\'an, Antonio},
	doi = {10.1090/tran/8950},
	fjournal = {Transactions of the American Mathematical Society},
	issn = {0002-9947,1088-6850},
	journal = {Trans. Amer. Math. Soc.},
	mrclass = {03D45 (03C75)},
	mrnumber = {4630765},
	mrreviewer = {Dino\ Rossegger},
	number = {8},
	pages = {5989--6008},
	title = {The {$\omega$}-{V}aught's conjecture},
	url = {https://doi.org/10.1090/tran/8950},
	volume = {376},
	year = {2023}}

@unpublished{Part2,
	author = {Antonio Montalb\'an},
	label = {Part 2},
	note = {In preparation},
	title = {Computable structure theory: Beyond the arithmetic},
	year = {P2}}

@book{Hod93,
	author = {Hodges, Wilfrid},
	doi = {10.1017/CBO9780511551574},
	isbn = {0-521-30442-3},
	mrclass = {03-01 (03-02 03Cxx)},
	mrnumber = {1221741 (94e:03002)},
	mrreviewer = {J. M. Plotkin},
	pages = {xiv+772},
	publisher = {Cambridge University Press, Cambridge},
	series = {Encyclopedia of Mathematics and its Applications},
	title = {Model theory},
	url = {http://dx.doi.org/10.1017/CBO9780511551574},
	volume = {42},
	year = {1993}}

@article{Kru60,
	author = {Kruskal, J. B.},
	doi = {10.2307/1993287},
	fjournal = {Transactions of the American Mathematical Society},
	issn = {0002-9947,1088-6850},
	journal = {Trans. Amer. Math. Soc.},
	mrclass = {06.00},
	mrnumber = {111704},
	mrreviewer = {Graham\ Higman},
	pages = {210--225},
	title = {Well-quasi-ordering, the {T}ree {T}heorem, and {V}azsonyi's conjecture},
	url = {https://doi.org/10.2307/1993287},
	volume = {95},
	year = {1960}}

@article{Ric81,
	author = {Richter, Linda Jean},
	coden = {JSYLA6},
	fjournal = {The Journal of Symbolic Logic},
	issn = {0022-4812},
	journal = {J. Symbolic Logic},
	mrclass = {03C57 (03D30)},
	mrnumber = {MR641486 (83d:03048)},
	mrreviewer = {V. K. Bulitko},
	number = {4},
	pages = {723--731},
	title = {Degrees of structures},
	volume = {46},
	year = {1981}}

@article{McountingBF,
	author = {Antonio Montalb\'an},
	journal = {Journal of Logic and Computability},
	number = 4,
	pages = {857--876},
	title = {Counting the back-and-forth types},
	url = {numberbftypesCorrected.pdf},
	volume = 22,
	year = {2010}}

@article{Kni86,
	author = {Knight, Julia F.},
	coden = {JSYLA6},
	doi = {10.2307/2273915},
	fjournal = {The Journal of Symbolic Logic},
	issn = {0022-4812},
	journal = {J. Symbolic Logic},
	mrclass = {03D30 (03C15)},
	mrnumber = {865929},
	mrreviewer = {R. A. Di Paola},
	number = {4},
	pages = {1034--1042},
	title = {Degrees coded in jumps of orderings},
	url = {http://dx.doi.org/10.2307/2273915},
	volume = {51},
	year = {1986}}

@article{GR24,
	author = {David Gonzalez and Dino Rossegger},
	journal = {Journal Symbolic Logic},
	label = {GR},
	note = {to appear},
	title = {The {S}cott Sentence Complexity of Linear Orders},
	year = {2024}}

@unpublished{GLRS,
	author = {David Gonzalez and Mateusz Lelyk and Dino Rossegger and Patryk Szlufik},
	note = {Preprint available at https://arxiv.org/abs/2507.12025},
	title = {Classifying the Complexities of Models of Arithmetic}}

@article{HT18,
	author = {Harrison-Trainor, Matthew},
	doi = {10.1016/j.aim.2018.03.012},
	fjournal = {Advances in Mathematics},
	issn = {0001-8708,1090-2082},
	journal = {Adv. Math.},
	mrclass = {03D45 (03C57 03C75 03D60 03E60)},
	mrnumber = {3787542},
	mrreviewer = {Wesley\ Calvert},
	pages = {109--147},
	title = {Scott ranks of models of a theory},
	url = {https://doi.org/10.1016/j.aim.2018.03.012},
	volume = {330},
	year = {2018}}
\end{document}